\documentclass[11pt, reqno, a4paper,oneside]{amsart}
\usepackage{amsmath,amssymb,amsfonts}
\usepackage[T1]{fontenc}
\usepackage{color}
\usepackage{float}
\usepackage{hyperref}
\usepackage{graphicx}
\usepackage{mathtools}
\usepackage{ulem}
\newcommand{\mygraphic}[1]{\includegraphics[height=#1]{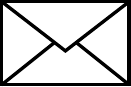}}
\newcommand{\myenv}{(\raisebox{0pt}{\mygraphic{.6em}})}

\usepackage{epstopdf}
\usepackage{algorithmic}
\ifpdf
  \DeclareGraphicsExtensions{.eps,.pdf,.png,.jpg}
\else
  \DeclareGraphicsExtensions{.eps}
\fi

\numberwithin{equation}{section}

\newtheorem{defin}{Definition}[section]
\newtheorem{theorem}[defin]{Theorem}
\newtheorem{lemma}[defin]{Lemma}

\theoremstyle{definition} {\newtheorem{remark}[defin]{Remark}}

\usepackage{amsopn}
\DeclareMathOperator{\diam}{diam}

\newcommand{\smfrac}[2]{{\textstyle \frac{#1}{#2}}}

\def\b{\big}
\def\B{\Big}
\def\bg{\bigg}

\def\bsep{\,\b|\,}
\def\Bsep{\,\B|\,}

\def\R{\mathbb{R}}

\def\N{\mathbb{N}}

\def\dist{\mathrm{dist}}

\DeclareMathOperator*{\argmax}{{\rm argmax}}

\def\CC{\mathrm{C}}

\newcommand{\lo}{\lambda_\Omega}
\newcommand{\Lo}{\Lambda_\Omega}

\def\<{\langle}
\def\>{\rangle}

\def\dd{\,{\rm d}}

\def\dt{\,{\rm d}t}

\def\ds{\,{\rm d}s}

\def\AXint#1#2#3{{\setbox0=\hbox{$#1{#2#3}{\int}$}
\vcenter{\hbox{$#2#3$}}\kern-.5\wd0}}

\def\E{\mathcal{E}}

\def\Om{\Omega}
\def\Dom{\mathcal{D}}

\def\rhobar{{\bar{\rho}}}

\title[Dislocation dynamics: time estimates on collisions]{Properties of screw dislocation dynamics: time estimates on boundary and interior collisions}
\author{Thomas Hudson}
\address{Mathematics Institute, Zeeman Building, University of Warwick, Coventry, CV4 7AL, United Kingdom}
\email[T.~Hudson]{t.hudson.1@warwick.ac.uk}
\author{Marco Morandotti}
\address{Fakult\"at f\"ur Mathematik, Technische Universtit\"at M\"unchen, Boltzmannstrasse, 3, 85748 Garching, Germany}
\email[M.~Morandotti~\myenv]{marco.morandotti@ma.tum.de}

\date{7 March 2017}
\begin{document}

\maketitle

\begin{abstract}
In this paper, the dynamics of a system of a finite number of screw dislocations is studied.
Under the assumption of antiplane linear elasticity, the two--dimensional dynamics is determined by the renormalised energy.
The interaction of one dislocation with the boundary and of two dislocations of opposite Burgers moduli are analysed in detail and estimates on the collision times are obtained.
Some exactly solvable cases and numerical simulations show agreement with the estimates obtained.
\end{abstract}


\smallskip



\tableofcontents

\section{Introduction}

Dislocations are topological line defects found in crystalline solids, and their motion governs the
plastic flow in such materials.
As a consequence, they are objects of great interest to materials scientists and engineers; despite
having been initially studied over a century ago \cite{V1907}, and having been proposed as the atomic mechanism for plasticity \cite{Orowan34,Polanyi34,Taylor34},
their collective behaviour remains a topic of ongoing research, both since they interact at long range
via the stress fields they induce in the crystal, and because of their inherent complexity as a network
of curves.

A variety of works in the mathematical literature have begun to address questions relating to dynamical models
of dislocation motion \cite{ADLGP14,ADLGP16,BFLM15,BM17,BvMM15,CEH10,FM09,GM10,MP12,vMM14}, and this paper contributes to that ongoing thread of research by studying various
properties of a model for the dynamics, in particular collisions, of straight screw dislocations in a long straight cylinder.
It is worth mentioning that the dynamics of screw dislocations has significant similarities to that of Ginzburg-Landau vortices in two dimensions \cite{BBH94,SS07}.
Properties of vortices up to collision time have been studied extensively (see, \emph{e.g.}, \cite{BOS05,JS98,L96,S07} and the references therein), but the results presented here are the first ones, to the best of our knowledge, to provide sharp estimates on the collision times for dislocations.

The model we consider was first proposed in full generality in \cite{CG99}, and studied extensively with specific choices of mobility in
\cite{BFLM15,BvMM15}. 
In particular, the latter two works prove existence of the evolution by differing methods, but acknowledge that blow--up of solutions appears to be a ubiquitous phenomenon. 
Generically, this blow--up seems to occur either via a collision between two dislocations, or the collision of a dislocation and the boundary. 
Here, we rigorously analyse three properties of the model proposed in \cite{CG99}, namely we analytically investigate
\begin{itemize}
\item[(i)] the behaviour of dislocations near a free boundary, 
\end{itemize}
and the blow--up of solutions in detail when: 
\begin{itemize}
\item[(ii)] one dislocation collides with the boundary; 
\item[(iii)] two dislocations collide with one another.
\end{itemize}
In particular, we provide geometric conditions on the initial configuration of dislocations to ensure that blow--up occurs due to (ii) or (iii).
In so doing, we give estimates on the blow--up time and establish a sufficient condition ensuring that no other collision events occur before this blow--up time.

\smallskip

Following the standard mechanical setting of \cite{CG99}, we consider a finite number of idealised screw dislocations in an infinite cylinder $\Omega\times\R$ undergoing an antiplane deformation.
Since the displacement only occurs in the vertical direction, this entails two major simplifications, namely, that the problem can be studied in the two-dimensional cross--section and that the \emph{Burgers vectors} measuring the lattice mismatch are indeed scalar quantities (the \emph{Burgers moduli} being directed along the vertical axis).

We assume that the cylindrical domain has a constant cross--section,
$\Om\subset\R^2$, which is an open connected set with a $\CC^2$ boundary. 
In particular, this regularity
assumption entails that the boundary satisfies \emph{uniform interior and exterior disk conditions}; 
that is, there exists $\rhobar>0$ such that
for any point $x\in\partial\Om$, there exist unique points
$x_\text{int}$ and $x_\text{ext}$ such that
\begin{equation}\label{1011}
  B_\rhobar(x_\text{int})\subseteq \Om,\quad
  B_\rhobar(x_\text{ext})\subset \Om^c\quad\text{and}\quad
  \partial B_\rhobar(x_\text{int})\cap\partial\Om\cap
  \partial B_\rhobar(x_\text{ext}) = \{x\},
\end{equation}
where $B_r(x)\subset\R^2$ is the open disk of radius $r$ centred at
$x$ and the superscript $c$ denotes the complement of a set. 
We shall fix such $\rhobar>0$ once and for all.
It also follows that the curvature of the boundary, $\kappa$, lies
in $\CC(\partial\Om)$, and $\|\kappa\|_{\infty}\leq\rhobar^{-1}$. 
An illustration of the interior and exterior disk conditions is contained in \eqref{iedc}.

\begin{figure}[h]
\begin{center}
\includegraphics[scale=.25]{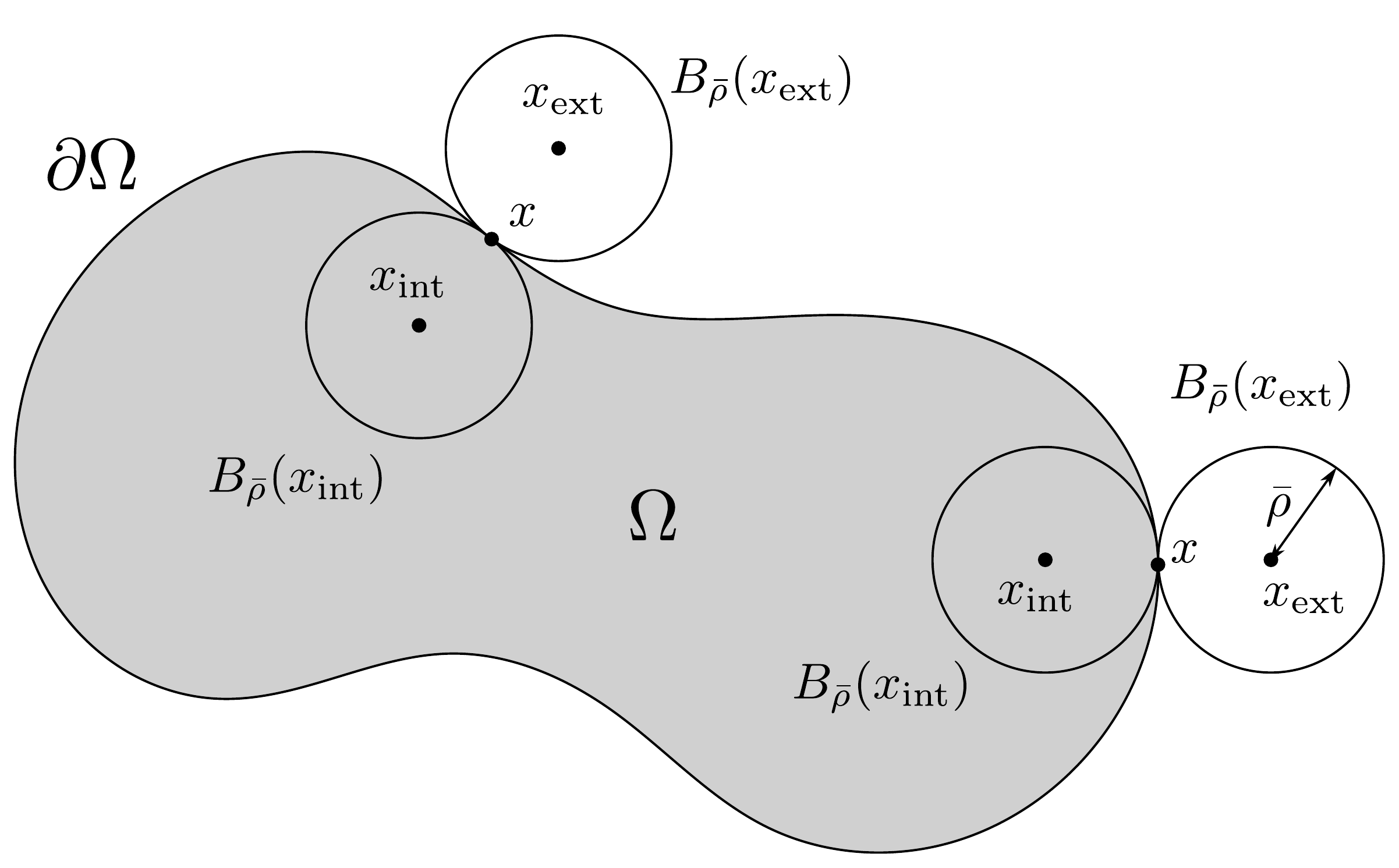}
\caption{The uniform interior and exterior disk condition, with the interior and exterior disks illustrated for two points.}
\label{iedc}
\end{center}
\end{figure}

%

\subsection{Renormalised energy and Peach--Koehler forces}
Since dislocations are topological singularities, the conventional
continuum elastic energy of the strain field induced by a dislocation
is infinite. This reflects the
fact that a dislocation is fundamentally a discrete phenomenon, which
is not adequately captured by continuum elasticity.
Instead, a process of \emph{renormalisation} can be used to provide
a potential energy for a collection of straight screw dislocations.
This process involves the subtraction of an `infinite' constant from
the potential energy, reflecting the fact that continuum elasticity
is simply an asymptotic description of matter in the case where large
variations occur on length--scales much greater than that of the lattice
spacing.
The resulting \emph{renormalised energy} appears in numerous studies of 
topological singularities, see for example \cite{ADLGP14,BBH94,CL05,SS03,SS07}. 

One approach to justifying the use of the renormalised energy is to
define it via the \emph{core-radius approach} as in \cite{BM17,CL05}. Here, we
proceed directly to a definition of the renormalised energy, given in
terms of Green's functions for the Laplacian on the domain $\Om$.

Define the Green's function of the Laplacian with Dirichlet boundary 
conditions on $\partial\Om$ as the (distributional) solution to
\begin{equation}\label{eq:G_Om_problem}
\begin{cases}
  -\Delta_x G_\Om(x,y) = \delta_y(x) & \text{in }\Om, \\
  G_\Om(x,y) = 0 & \text{on }\partial\Om.
\end{cases}
\end{equation}
Here $\delta_y$ is the usual Dirac delta distribution centred
at a point $y\in\Om$.
We emphasise the domain $\Om$ on which this function is defined, since we will later vary this domain in order to obtain our estimates.
It is a classical result that $G_\Om$ is smooth in the variable $x$ on the set $\Om\setminus\{y\}$ for any
given $y\in\Om$; is symmetric, i.e. $G_\Om(x,y) = G_\Om(y,x)$; and 
\begin{equation}\label{eq:Green}
  G_\Om(x,y) = -\frac1{2\pi}\log|x-y|+ k_\Om(x,y),
\end{equation}
where $k_\Om(x,y)$ is smooth in both arguments on $\Om$, is symmetric, and satisfies the elliptic boundary value problem
\begin{equation}\label{eq:k_Om_problem}
\begin{cases}
  -\Delta_x k_\Om(x,y) = 0 & \text{in }\Om, \\
  k_\Om(x,y) = \frac1{2\pi}\log|x-y| &\text{on }\partial\Om;
\end{cases}
\end{equation}
proofs of all of the above assertions may be found in Chapter~4 of \cite{H14}.
In addition, we also define 
\begin{equation}\label{eq:h_Om}
h_\Om(x):=k_\Om(x,x),
\end{equation}
which will turn out to be a convenient function with which to express the renormalised energy.
By exploiting conformal transformations in $\R^2$, it can be shown that $h_\Om$ satisfies the elliptic problem (see \cite{CF85}, Exercise~1, p548 of \cite{Friedman}, and \cite{G90})
\begin{equation}
  -\Delta_x h_\Om(x) = \smfrac2\pi \mathrm{e}^{-4\pi h_\Om(x)}\quad
  \text{for all }x\in\Om.\label{eq:h_Om_problem}
\end{equation}
Its properties will be studied in Section \ref{sect:estimates}.

Using the explicit expression for the Green's function \eqref{eq:Green} and the functions defined in \eqref{eq:k_Om_problem} and \eqref{eq:h_Om}, the renormalised energy of $n$ dislocations with positions $z_1,\ldots,z_n\in\Om$ and Burgers moduli $b_1,\ldots,b_n\in\{-1,+1\}$ (see, \emph{e.g.}, \cite{ADLGP14,BM17,CG99}) may be expressed as
\begin{equation}\label{renen}
  \E_n(z_1,\ldots,z_n):= 
  \sum_{i<j}b_ib_j\bigg(k_\Omega(z_i,z_j)-\frac1{2\pi}\log|z_i-z_j|\bigg) + \frac12 \sum_{i=1}^n b_i^2 h_\Omega(z_i),
\end{equation}
where the contributions of the two--body
interaction terms and the one--body `self--interaction' term are highlighted.
To be more precise, each term $h_\Om(z_i)$ is the contribution to the energy given by a single dislocation
sitting at $z_i$; the logarithmic terms $\log|z_i-z_j|$ account for the interaction energy of the two
dislocations sitting at $z_i$ and $z_j$; the term $k_\Om(z_i,z_j)$ accounts for the interaction of the
dislocation sitting at $z_i$ with the boundary response due to the dislocation sitting at $z_j$.
It is also worth mentioning that the interaction terms also involve the product $b_ib_j$ of the Burgers
moduli of the dislocations in a fashion similar to that of electric charges: $b_i=b_j=\pm1$ gives a positive
contribution to the energy, and tends to push two dislocations with the same sign far away from each other.
This will become clearer in the expression of the force, responsible for the motion, acting on the
dislocations.
Finally, notice that the terms with the subscript $\Om$ depend in a crucial way on the geometry of the domain
and carry information about the interaction with the boundary.
To be thorough, the energy $\E_n$ defined in \eqref{renen} should also depend on the Burgers moduli $b_1,\ldots,b_n$, but we assume these are attached to the dislocations and do not vary in time, so we suppress this dependence in the interest of concision.

The force acting on a dislocation is the so-called Peach--Koehler force \cite{HL1982} and it is obtained by taking the negative of the gradient with respect to the dislocation position
\begin{equation}\label{502}
f_i(z_1,\ldots,z_n)=-\nabla_{z_i}\E_n(z_1,\ldots,z_n),\qquad\text{for $i=1,\ldots,n$.}
\end{equation}
The subscript $i$ refers to the force experienced by the dislocation at $z_i$ and the dependence on the whole configuration of dislocations $z_1,\ldots,z_n$ highlights the nonlocal character of the Peach--Koehler force.

The law describing the dynamics of the dislocations is therefore expressed as
\begin{equation}\label{500}
\dot z_i(t)=-\nabla_{z_i}\E_n(z_1(t),\ldots,z_n(t)),\qquad\text{for $i=1,\ldots,n$,}
\end{equation}
complemented with suitable initial condition at time $t=0$.
Formula \eqref{500} usually includes a \emph{mobility} function, which here we have taken equal to the identity.
Various suggestions for possible mobility functions can be found in \cite{CG99}; we refer the reader to Section \ref{conclusion} for a discussion on other possible choices that are relevant in our context.
For a specific choice of the mobility, \eqref{500} takes the form of a differential inclusion, and was studied both in \cite{BFLM15} to obtain existence and uniqueness results, and in \cite{BvMM15} from the point of view of gradient flows.

\subsection{Aims}
Our results below rigorously verify a variety of qualitative features
of \eqref{500} for the dynamics of dislocations. It is commonly
observed in numerical simulations that dislocations are attracted to free boundaries and 
that dislocations of opposite signs attract. In fact, as
dislocations approach the boundary, or as 
dislocations with Burgers moduli of opposite sign approach one
another, the renormalised energy diverges to $-\infty$, and hence
solutions of the evolution problem blow up and cease to exist, at least in the
senses considered in \cite{BFLM15,BvMM15}.

In Lemma \ref{th:h_derivative_bound}, we prove a gradient bound for the function $h_\Om$ for points in the vicinity of the boundary: this allows us to treat case (i). 
We prove Theorem \ref{thm:fatal}, which states that the main component of the Peach--Koehler force on a dislocation close to the boundary is directed along the outward unit normal at the boundary point closest to the dislocation, thus demonstrating that free boundaries attract dislocations. 
The result is obtained by characterising the Peach--Koehler force acting on a dislocation sufficiently close to the boundary up to an error which is uniformly bounded in various geometric parameters of the system, namely the mutual distances of the dislocations and the curvature of the boundary.

In Theorem \ref{thm:boundary_collision}, we address (ii): we consider the situation of a dislocation near the boundary, well separated from all the others.
The result that we obtain is an upper bound for the collision time and an estimation of how close to the boundary this dislocation must be in order to collide with it before any other collision event.

In Theorem \ref{thm:collision}, we turn to (iii) where two dislocations of opposite Burgers moduli are close to each other and well separated from the others.
In this case, we again obtain an upper bound for the collision time and conditions on the geometry of the initial configuration which guarantee that no other collision events occur before the two dislocations hit one another.

In both Theorems \ref{thm:boundary_collision} and \ref{thm:collision}, the geometric conditions obtained are invariant under dilation of the coordinate system, but, whereas those needed for Theorem \ref{thm:boundary_collision} explicitly involve the curvature of the domain, those needed for Theorem \ref{thm:collision} only depend on the domain through its diameter (therefore, the regularity of the boundary is not relevant for the latter result).

While these behaviours are expected from a qualitative point of view, the novelty of our results is that sharp estimates on the collision times are provided for the first time.
Moreover, the interaction with the boundary characterised in Theorem \ref{thm:fatal} and the estimates of Theorems \ref{thm:boundary_collision} and \ref{thm:collision} are determined in a scale-invariant way in terms of geometric parameters describing the shape of the domain (through its curvature) and the configuration of the dislocations.
It is worth mentioning that the geometry of the domain and the arrangement of the dislocations are only responsible for the higher order corrections to estimates on the collision times.

\smallskip

The paper is organised as follows: in Section \ref{sect:estimates} we provide some estimates on the functions $G_\Om$
and $h_\Om$ that will be crucial for the rest of the paper.
In Section \ref{sect:attraction}, we state and prove the main theorems about the attracting behaviour of free boundaries, and the estimates on the collision times of one dislocation with the boundary and of two dislocations hitting each other.
These results will be compared in Section \ref{sect:collision} to some explicit cases also discussed in \cite{BFLM15}.
We also include numerical plots for domains (namely the square and the cardioid) which exhibit interesting symmetries.
Finally, in Section \ref{conclusion} we draw some conclusions and discuss other models for the relationship between the Peach--Koehler force and the velocity of dislocations.
In Appendix \ref{appendix} we collect some explicit expressions for the Green's functions for the disc and its exterior used in our analysis.

\section{Preliminaries and estimates of interaction kernels}\label{sect:estimates}

In this section, we collect the series of asymptotic bounds on Green's
functions and related interaction kernels which we will require in our
analysis. A particular focus will be asymptotics for the gradients of these
kernels, since these provide a description of the Peach--Koehler forces acting on dislocations.

An important function in what follows will be $d_n\colon\Omega^n\to[0,+\infty)$, 
which is defined as:
\begin{equation}\label{1001}
  d_n(x_1,\ldots,x_n):=
  \begin{cases}
    \dist(x_1,\partial\Om) & n=1,\\
    \min_i\dist(x_i,\partial\Om)\wedge \min_{i\neq j}|x_i-x_j| &\text{otherwise}.
  \end{cases}
\end{equation}
In the case $n=1$, the function $d_1$ measures the distance of the dislocation from the boundary $\partial\Omega$, whereas, if $n\geq2$, $d_n$ describes the minimal separation among the dislocations and their distance from the boundary.
As will be clear in the sequel, this function arises as the distance between a configuration and some critical set in $\Omega^n$ on which the evolution \eqref{500} ceases to exist.

\smallskip

We now use the descriptions of $G_\Om$, $k_\Om$, and $h_\Om$ as respective
solution of the elliptic problems \eqref{eq:G_Om_problem},
\eqref{eq:k_Om_problem}, and \eqref{eq:h_Om_problem}
along with the comparison principle in order to provide asymptotic
gradient estimates for these functions in a variety of
situations. A key tool will be the following bound, taken from Section~3.4 in \cite{GT}:
let $f\in\CC^0(B_r(0))$ and let
$u\in\CC^2(B_r(0))\cap\CC^0\b(\overline{B_r(0)}\b)$ satisfy $-\Delta u=f$
in $B_r(0)\subset\R^2$; then
\begin{equation}
  |\nabla u(0)|\leq \frac2r\sup_{\partial B_r(0)}|u|
  +\frac{r}{2}\sup_{B_r(0)}|f|.\label{eq:gradient_bound}
\end{equation}

\subsection{Estimates on $\nabla G_\Om$}
Our first results concern an estimate on the boundary behaviour of the gradient of $G_\Om$ and some asymptotic formulae which will be relevant when two dislocations of opposite sign approach one another.

\begin{lemma}
\label{th:G_derivative_bounds}
Suppose that $\Om\subset\R^2$ is a $\CC^2$ domain, that $y\in\Om$ is fixed, and let $\rhobar>0$ be as in \ref{1011}; recall also the definition of $d_1(\cdot)$ from \ref{1001} for $n=1$.
Then:
\begin{enumerate}
\item for $x$ satisfying $d_1(x)<\rhobar$ and $d_1(x)<|x-y|$, we have
\begin{equation}
  |\nabla_x G_\Om(x,y)|\leq \frac{2\b(|y-x_\mathrm{ext}|^2
    -\rhobar^2\b)\b(\rhobar+d_1(x)\b)}{\pi\rhobar^2\b(|x-y|-d_1(x)\b)^2},
    \label{eq:dxG_bound}
\end{equation}
where $B_\rhobar(x_\mathrm{ext})$ is the disk which touches $\partial\Om$
at the point closest to $x$ on the boundary, and
\item for any $x\in\Om$, we have the bound
\begin{equation}
  |\nabla_y G_\Om(x,y)|\leq \frac{1}{2\pi|x-y|}+\frac{1}
    {2\pi d_1(y)}.\label{dyG_bound}
\end{equation}
\end{enumerate}
\end{lemma}
\begin{proof}
Let $\Om\subset\Om'$ and $y\in\Om$; then by the comparison principle, we have the
upper and lower estimates
\begin{equation*}
  G_{\Om'}(x,y) \geq G_\Om(x,y) \geq 0\quad\text{for all }x\in\Om.
\end{equation*}
To prove assertion (1), we note that if $d_1(x)<\rhobar$, there exists a unique point
$s\in\partial\Om$ such that $|x-s|=d_1(x)$, and $d_1(x)$ is a $\CC^2$ function
(see Lemma~14.16 of \cite{GT}, or \cite{KP81}) on this neighbourhood of $\partial\Om$.
Moreover,
\begin{equation*}
  \nabla d_1(x) = -\nu(s),
\end{equation*}
where $\nu(s)$ is the outward--pointing unit normal to
$\partial\Om$ at $s$. Since $\Om$ satisfies an exterior disk condition, there exists
$x_\mathrm{ext}\in\Om^c$ such that 
$B_\rhobar(x_\mathrm{ext})\subset\Om^c$ and
$s\in\partial B_\rhobar(x_\mathrm{ext})\cap \partial\Om$.
Using the explicit expression for the Green's function on
$B_\rhobar(x_\mathrm{ext})^c$ derived in Section \ref{sec:disk_funcs},
we therefore deduce that
\begin{equation*}
  0\leq G_\Om(x,y)\leq G_{B_\rhobar(x_\mathrm{ext})^c}(x,y)
  =\frac{1}{4\pi}\log\bg(1+\frac{\b[|y-x_\mathrm{ext}|^2-\rhobar^2\b]
    \b[2d_1(x)\rhobar+d_1(x)^2\b]}{\rhobar^2|x-y|^2}\bg).
\end{equation*}
Now, since $d_1(x)\leq |x-y|$ by assumption, $G_\Om$ is harmonic in $B_{d_1(x)}(x)$, and
we use \eqref{eq:gradient_bound} on $\Om$ with $f=0$ followed by the elementary inequality
$\log(1+r)\leq r$ for any $r\geq 0$ to deduce that
\begin{align*}
  \b|\nabla_xG_\Om(x,y)\b| &\leq \frac{2}{d_1(x)}
  \sup_{z\in\partial B_{d_1(x)}(x)}\b|G_\Om(z,y)\b|\\
  &\leq\frac{1}{2\pi d_1(x)}\sup_{z\in\partial B_{d_1(x)}(x)}\log\bg(1+\frac{\b[|y-x_\mathrm{ext}|^2
    -\rhobar^2\b]\b[2d_1(z)\rhobar+d_1(z)^2\b]}{\rhobar^2|z-y|^2}\bg),\\
  &\leq \frac{2\b(|y-x_\mathrm{ext}|^2
    -\rhobar^2\b)\b(\rhobar+d_1(x)\b)}{\pi\rhobar^2\b(|x-y|-d_1(x)\b)^2}.
\end{align*}
To obtain the final line, we estimate $d_1(z)\leq 2d_1(x)$ in the numerator and $|z-y|\geq |x-y|-d_1(x)$ in the
denominator.

Turning to assertion (2), we note that for fixed $y\in\Om$,
$\nabla_y G_\Om(\cdot,y)$ satisfies
\begin{equation*}
  \nabla_y G_\Om(x,y) = \frac{x-y}{2\pi|x-y|^2}+\nabla_y k_\Om(x,y),
\end{equation*}
where, with the Laplacian acting on each coordinate, we have
\begin{equation*}
  -\Delta_x\nabla_y k_\Om(x,y) = 0\text{ in }\Om\quad\text{and}\quad
  \nabla_y k_\Om(x,y) = \frac{y-x}{2\pi|y-x|^2}\text{ on }\partial\Om.
\end{equation*}
Applying the maximum principle, we therefore find that
\begin{equation*}
  \b|\nabla_y G_\Om(x,y)\b| \leq \frac{1}{2\pi|x-y|}
  +\sup_{s\in\partial\Om} \b|\nabla_y k_\Om(s,y)\b| \leq
  \frac{1}{2\pi|x-y|}+\frac{1}{2\pi d_1(y)},
\end{equation*}
as required.
\end{proof}
\begin{remark}
Notice that estimate \eqref{eq:dxG_bound} deteriorates if the points $x$ and $y$ become too close.
We stress here that we will use \eqref{eq:dxG_bound} only in the case $d_1(x)\ll|x-y|$.
\end{remark}

\subsection{Estimates on $\nabla h_\Om$}
We now provide estimates for the function $\nabla h_\Omega(x)$ both in the case when $x$ close to and far from the boundary.
We start with the latter.
\begin{lemma}\label{esthfar}
Let $x\in\Omega$ and define $\lo:=|\log(\diam\Om/2)|$.
Then
\begin{equation}\label{estnablah}
|\nabla h_\Om(x)|\leq\frac{2\max\{-\log d_1(x),\lo\}}{\pi d_1(x)}.
\end{equation}
\end{lemma}
\begin{proof}
Observe that by the symmetry of $k_\Om$, we can write
$$\nabla_x h_\Om(x)=\nabla_x k_\Om(x,x)=2\nabla_x k_\Om(x,y)|_{y=x}\,.$$
Since $k_\Om$ solves \eqref{eq:k_Om_problem}, by using \eqref{eq:gradient_bound} with $r=d_1(x)$ and the maximum principle, we obtain
\begin{equation*}
\begin{split}
|\nabla_x k_\Om(x,y)||_{y=x}\leq & \frac{2}{d_1(x)}\sup_{s\in\partial B_r(x)}|k_\Om(s,x)|\leq  \frac{2}{d_1(x)}\sup_{s\in\partial\Om}|k_\Om(s,x)| \\
= &  \frac{1}{\pi d_1(x)}\sup_{s\in\partial\Omega}\b|\log|s-x|\b|\leq \frac{\max\{-\log d_1(x),|\log(\diam\Omega/2)|\}}{\pi d_1(x)},
\end{split}
\end{equation*}
from which the thesis follows.
\end{proof}

As we will see shortly, the interaction between a dislocation and the boundary
can be expressed using the function $h_\Om$: we therefore prove the
following asymptotic description of $\nabla h_\Om$ near the boundary, following
the method of \cite{CF85}. 
The result presented below is sharper than that obtained
in this previous work, since we obtain a uniform bound which depends on the geometry
of the domain: this extra detail will be important for our subsequent analysis of the dynamics of
\eqref{500}.

\begin{lemma}
  \label{th:h_derivative_bound}
Suppose that $\Om$ is $\CC^2$ and satisfies interior and exterior
disk conditions with radius $\rhobar$. 
Then, for any $\sigma\in(0,1)$, 
if $d_1(x)\leq \sigma\rhobar$, there exists a constant $C_\sigma>0$ (depending only on $\sigma$) such that
\begin{equation}
  \bg|\nabla h_\Om(x)+\frac{\nu(s)}{2\pi d_1(x)}\bg|\leq
  \frac{C_\sigma}{\pi\rhobar},\label{eq:dh_bound}
\end{equation}
where $s\in\partial\Omega$ is the point which realises the distance to the boundary.
\end{lemma}

\begin{proof}
We recall that $h_\Om$ satisfies \eqref{eq:h_Om_problem}, and therefore
elliptic regularity theory implies that $h_\Om$ is smooth in $\Om$.
Moreover, by employing the maximum principle and the fact that
the right--hand side of \eqref{eq:h_Om_problem} is positive and
decreasing, we find that
\begin{equation}
  \Om\subseteq \Om' \quad\text{implies}\quad
  h_{\Om}\leq h_{\Om'}\quad\text{in }\Om.\label{eq:h_comparison}
\end{equation}
This fact will now allow us to construct estimates similar to
those found in Section~2 of \cite{CF85} by using explicit 
expressions for $h_\Om$ when $\Om$ is the interior or exterior
of a ball.

Recalling that $\Om$ satisfies \eqref{1011}, 
using the comparison principle \eqref{eq:h_comparison} and the
expressions for $h_{B_\rhobar(0)}$ and $h_{B_\rhobar(0)^c}$ derived in Section \ref{sec:disk_funcs}, it follows that
\begin{equation}
  \frac1{2\pi}\log\left(2 d_1(x)-\frac{d_1(x)^2}\rhobar\right)
  \leq h_\Om(x) \leq \frac1{2\pi}\log\left(2d_1(x)+\frac{d_1(x)^2}\rhobar\right)\quad
  \text{in }B_\rhobar(x_\mathrm{int}).\label{eq:hOm_bounds}
\end{equation}
Subtracting $\smfrac1{2\pi}\log|2d_1(x)|$, we obtain
\begin{equation}\label{911}
  \frac1{2\pi}\log\left(1-\frac{d_1(x)}{2\rhobar}\right)
  \leq h_\Om(x)-\frac1{2\pi}\log2 d_1(x) \leq 
  \frac1{2\pi}\log\left(1+\frac{d_1(x)}{2\rhobar}\right)\quad\text{in }
  B_\rhobar(x_\mathrm{int}),
\end{equation}
since $B_\rhobar(x_\mathrm{int})\subset\Omega\subset B_\rhobar(x_\mathrm{ext})^c$.
Notice that for any $\ell\in(0,1)$, we can estimate $|\log(1+t)|\leq|\log(1-\ell)||t|/\ell$, if $|t|\leq \ell$.
By applying this to \eqref{911} with $t=-d_1(x)/2\rhobar$ and $\ell=1/2$, it follows that 
\begin{equation}\label{eq:h_bound}
  \left|h_\Om(x)-\frac1{2\pi}\log2d_1(x)\right|\leq \frac{\log2}{2\pi}\frac{d_1(x)}\rhobar.
\end{equation}

Differentiating $h_\Om(x)-\frac1{2\pi}\log2 d_1(x)$, applying \eqref{eq:h_Om_problem} and the lower bound from \eqref{eq:hOm_bounds}, we find that
\begin{align*}
  -\Delta\bigg[h_\Om(x) & -\frac1{2\pi}\log2 d_1(x)\bigg] = 
  \frac2\pi\mathrm{e}^{-4\pi h_\Om(x)}
  +\frac1{2\pi}\Delta [\log2d_1(x)]\\
    &\leq\frac{1}{2\pi d_1^{2}(x)}\left(1-\frac{d_1(x)}{2\rhobar}\right)^{-2}
      +\frac{1}{2\pi}\left(\frac{\Delta d_1(x)}{d_1(x)}
      -\frac{|\nabla d_1(x)|^2}{d_1^{2}(x)}\right) \\
      &=\frac{1}{2\pi d_1^{2}(x)}\left[\left(\frac{2\rhobar}{2\rhobar-d_1(x)}\right)^2+d_1(x)\Delta d_1(x)-|\nabla d_1(x)|^2\right].
\end{align*}
Recalling that, when $d_1(x)\leq\rhobar$, $\nabla d_1(x) = -\nu(s)$,
where $\nu(s)$ is the outward--pointing unit normal at the boundary point $s$ which
is closest to $x$, we have $|\nabla d_1(x)|=1$.
Therefore,
the estimate above reads
\begin{equation}\label{912} 
-\Delta\left[h_\Om(x)-\frac1{2\pi}\log2 d_1(x)\right]\leq \frac{\Delta d_1(x)}{2\pi d_1(x)}+\frac1{2\pi d_1(x)}\frac{4\rhobar-d_1(x)}{(2\rhobar-d_1(x))^2}.
\end{equation}
We now estimate the two summands in the right-hand side above separately.
By recalling \cite[Lemma 14.17]{GT}, we have
$$\Delta d_1(x) = \frac{-\kappa(s)}{1-\kappa(s) d_1(x)},$$
where $\kappa(s)$ is the curvature at $s\in\partial\Om$ which realises $|s-x|=d_1(x)$; recalling that $d_1(x)\leq \sigma\rhobar$, we can estimate
$$\bg|\frac{\kappa(s)}{1-\kappa(s) d_1(x)}\bg|\leq \frac{1}{(1-\sigma)\rhobar}.$$
Noting that the map $[0,\sigma\rhobar]\ni t\mapsto (4\rhobar-t)/(2\rhobar-t)^2$ is increasing and it attains its maximum when $t=\sigma\rhobar$, \eqref{912} reads
\begin{equation}\label{913}
-\Delta\left[h_\Om(x)-\frac1{2\pi}\log2 d_1(x)\right]\leq \frac{1}{2\pi\rhobar d_1(x)}\frac{2\sigma^2-9\sigma+8}{(1-\sigma)(2-\sigma)^2}=:\frac{c_\sigma}{2\pi\rhobar d_1(x)}.
\end{equation}
Applying \eqref{eq:gradient_bound} on a ball centred at $x$ of radius $r=d_1(x)$, taking \eqref{eq:h_bound} and \eqref{913} into account, we obtain
\begin{equation*}
  \left|\nabla\left[h_\Om(x)-\frac1{2\pi}\log|2d_1(x)|\right]\right| \leq
  \frac{\log2}{\pi\rhobar}+\frac{c_\sigma}{4\pi\rhobar},
\end{equation*}
which is the thesis \eqref{eq:dh_bound} with 
\begin{equation}\label{Csigma}
C_\sigma:=\log2+\frac{c_\sigma}4=\log2+\frac{2\sigma^2-9\sigma+8}{4(1-\sigma)(2-\sigma)^2}.
\end{equation}
The lemma is proved.
\end{proof}

\section{Main results}\label{sect:attraction}
In this section, we prove our main results.
We will apply Lemma \ref{th:h_derivative_bound} first to study the Peach--Koehler force on a dislocation very close to the boundary, and then to obtain criteria on the initial conditions of the evolution such that dislocations hit the boundary or collide with each other within a given time interval.

\smallskip

The situation we consider in the next two subsections is the following: we suppose that we have $n\in\N$ dislocations in $\Omega$ one of which, $z_1$, is much closer to the boundary $\partial\Omega$ than the others; we also suppose that the other $n-1$ dislocations, $z_2,\ldots,z_n$, are spaced sufficiently far apart from each other and from the boundary.

We introduce the notation $z':=(z_2,\ldots,z_n)$ so that
the configuration of the $n$ dislocations can be represented by the vector $z:=(z_1,z')\in\Omega^n$.
Given $0<\delta<\gamma<\diam\Omega/2$, define the set
\begin{equation}\label{eq:Geomofz}
  \Dom_{n,\delta,\gamma}:=\b\{(z_1,z')\in\Om^n \bsep d_1(z_1)<\delta,d_{n-1}(z')>\gamma\b\}.
\end{equation}
The geometric meaning of the set $ \Dom_{n,\delta,\gamma}$ defined above is the following: if $z\in \Dom_{n,\delta,\gamma}$, it means that $z_1$ lies at a distance of at most $\delta$ from the boundary, while all the other dislocations $z_2,\ldots,z_n$ lie at a distance of at least $\gamma$ away from the boundary and their mutual distance is also at least $\gamma$.
The condition $\delta<\gamma$ ensured that $z_1$ is closer to the boundary than any other dislocation.

\subsection{Free boundaries attract dislocations}
In the following theorem we show that the Peach--Koehler force acting on a dislocation which is very close to the boundary is directed along the outward unit normal at the boundary point closest to the dislocation.
\begin{theorem}\label{thm:fatal}
Let $n\in\N$, let $\sigma\in(0,1)$, recall the definition of $\rhobar$ from \eqref{1011}, and let $\delta\in(0,\sigma\rhobar)$ and $\gamma\in(\max\{2\delta,\rhobar\},\diam\Omega/2)$.
Let $z=(z_1,z')\in\Dom_{n,\delta,\gamma}$.
Then, if $s\in\partial\Omega$ is the boundary point closest to $z_1$, the Peach-Koehler force $f_1(z)$ on the dislocation $z_1$ (see \eqref{502}) satisfies
\begin{equation}\label{505}
f_1(z)=\frac{\nu(s)}{4\pi d_1(z_1)}+O(1),
\end{equation}
where $O(1)$, which is quantified in \eqref{estfatal}, denotes a term which is uniformly bounded for all $z\in\Dom_{n,\delta,\gamma}$.
\end{theorem}
\begin{proof}
Since $z\in\Dom_{n,\delta,\gamma}$, $d_1(z_1)<\rhobar$, by the assumptions on $\partial\Omega$, there exists a unique point $s\in\partial\Omega$ such that $d_1(z_1)=|z_1-s|$.

Recalling \eqref{eq:Green}, we express the renormalised energy \eqref{renen} as
\begin{equation}\label{eq:E_1rest}
\E_n(z_1,\ldots,z_n) = \E_1(z_1)+\E_{n-1}(z')+\sum_{i=2}^nb_1b_i G_\Om(z_1,z_i),
\end{equation}
where we separate the contribution of $z_1$ and that of $z'$. The Peach-Koehler force \eqref{502} on $z_1$ can therefore be written as
\begin{equation}\label{506}
\begin{split}
f_1(z)=-\nabla_{z_1} \E_n(z)= & -\nabla_{z_1} \E_1(z_1)-\sum_{i=2}^n b_1b_i\nabla_{z_1}G_\Omega(z_1,z_i) \\
= & -\frac12\nabla_{z_1} h_\Omega(z_1)-\sum_{i=2}^n b_1b_i\nabla_{z_1}G_\Omega(z_1,z_i).
\end{split}
\end{equation}
To prove \eqref{505}, we estimate the difference
\begin{equation}\label{507}
\left| f_1(z)-\frac{\nu(s)}{4\pi d_1(z_1)}\right|\leq \left| f_1(z)+\frac12\nabla_{z_1} h_\Omega(z_1)\right|+\frac12\left|\nabla_{z_1} h_\Omega(z_1)+\frac{\nu(s)}{2\pi d_1(z_1)}\right|
\end{equation}
Invoking \eqref{506}, we use \eqref{eq:dxG_bound} to estimate the first term in the right-hand side above by 
\begin{equation*}
\left| f_1(z)+\frac12\nabla_{z_1} h_\Omega(z_1)\right|\leq
\sum_{i=2}^n \frac{2\b(|z_i-x_\mathrm{ext}|^2-\rhobar^2\b)\b(\rhobar+d_1(z_1)\b)}{\pi\rhobar^2\b(|z_1-z_i|-d_1(z_1)\b)^2}.
\end{equation*}
Recalling \eqref{eq:Geomofz} and setting $a_i:=|z_1-z_i|-d_1(z_1)$ for $i=2,\ldots,n$, we apply the triangle inequality to obtain $|z_i-z_1|\geq |z_i-s|-|s-z_1| > \gamma-\delta$, which entails that
\begin{subequations}\label{33}
  \begin{eqnarray}
a_i & \geq & \gamma-2\delta, \label{est3}\\
\qquad|z_i-x_\mathrm{ext}| & \leq & |z_i-z_1|+|z_1-x_{\mathrm{ext}}|\leq |z_i-z_1|+d_1(z_1)+\rhobar=a_i+2\delta+\rhobar; \label{est2}
\end{eqnarray}
\end{subequations}
see \eqref{tri}(b) for an illustration of the geometry.
Using \eqref{est2} we estimate 
\begin{equation}\label{510}
\begin{split}
|z_i-x_\mathrm{ext}|^2-\rhobar^2&\leq  (a_i+2\delta+\rhobar)^2-\rhobar^2 \\
&\leq a_i^2+2a_i(2\delta+\rhobar)+4\delta(\delta+\rhobar).
\end{split}
\end{equation}
Applying \eqref{510} to bound the numerator and the definition of $a_i$ in the denominator, and then
using \eqref{est3}, now gives
\begin{equation}\label{est4}
\begin{split}
  \sum_{i=2}^n & \frac{2\b(|z_i-x_\mathrm{ext}|^2-\rhobar^2\b)\b(\rhobar+d_1(z_1)\b)}{\pi\rhobar^2\b(|z_1-z_i|-d_1(z_1)\b)^2} \\
  & \leq\sum_{i=2}^n\frac{2(\rhobar+\delta)}{\pi\rhobar^2}
  \frac{a_i^2+2a_i(2\delta+\rhobar)+4\delta(\delta+\rhobar)}{a_i^2}\\
& \leq\frac{2(\rhobar+\delta)(n-1)}{\pi\rhobar^2}\!\!\left(1+2\frac{\rhobar+2\delta}{\gamma-2\delta}+4\frac{\delta(\rhobar+\delta)}{(\gamma-2\delta)^2}\right).
\end{split}
\end{equation}
Now, collecting terms in \eqref{est4}, then applying \eqref{eq:dh_bound} and the hypothesis that
$\delta<\sigma\rhobar$, estimate \eqref{507} becomes
\begin{equation}\label{estfatal}
\left| f_1(z)-\frac{\nu(s)}{4\pi d_1(z_1)}\right|\leq \frac{2(1+\sigma)(n-1)\gamma(\gamma+2\rhobar)}{\pi\rhobar(\gamma-2\sigma\rhobar)^2}+\frac{C_\sigma}{2\pi\rhobar}=:\frac{C_{n,\sigma}(\gamma)}{2\pi\rhobar},
\end{equation}
where the constant $C_{n,\sigma}(\gamma):=C_\sigma+4(1+\sigma)(n-1)\gamma(\gamma+2\rhobar)/(\gamma-2\sigma\rhobar)^2$ only depends on the geometric parameter $\rhobar$, on $\sigma\in(0,1)$, and on how far all the other dislocations are from $z_1$ and from $\partial\Omega$.
This proves \eqref{505}.
\end{proof}
\begin{figure}[h]
\begin{center}
\includegraphics[scale=.35]{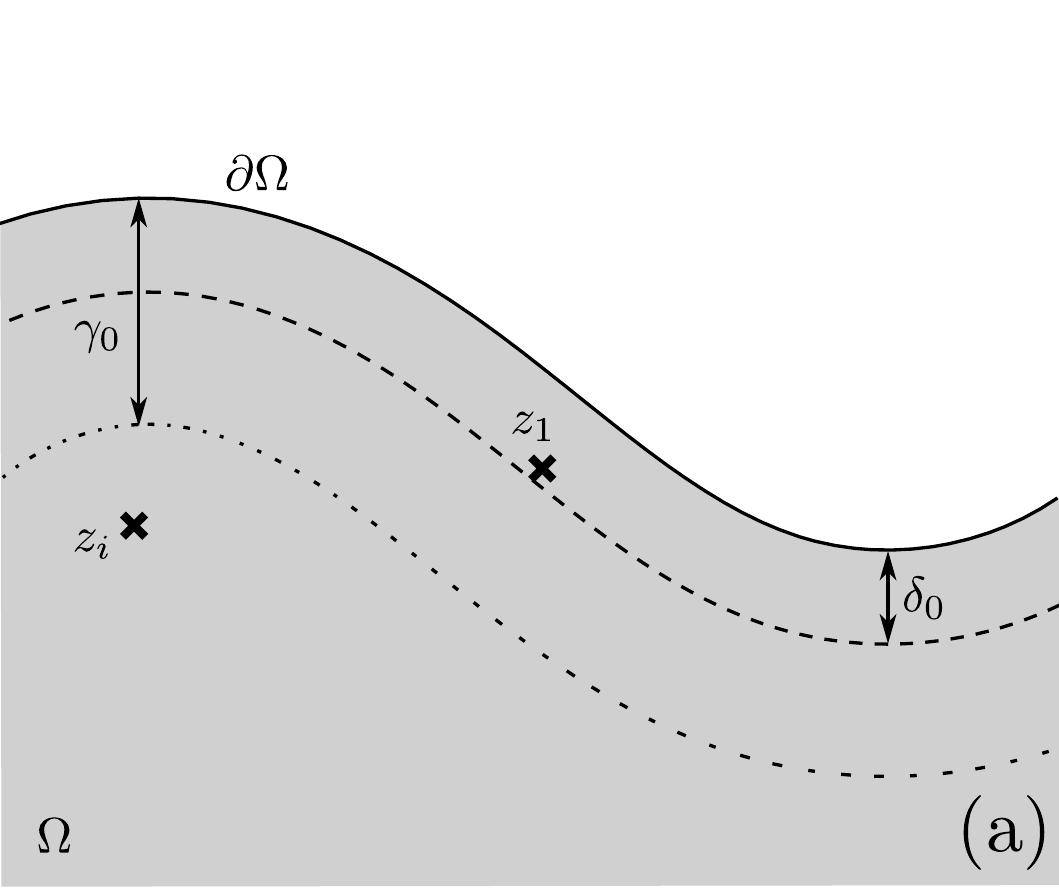}\quad
\includegraphics[scale=.47]{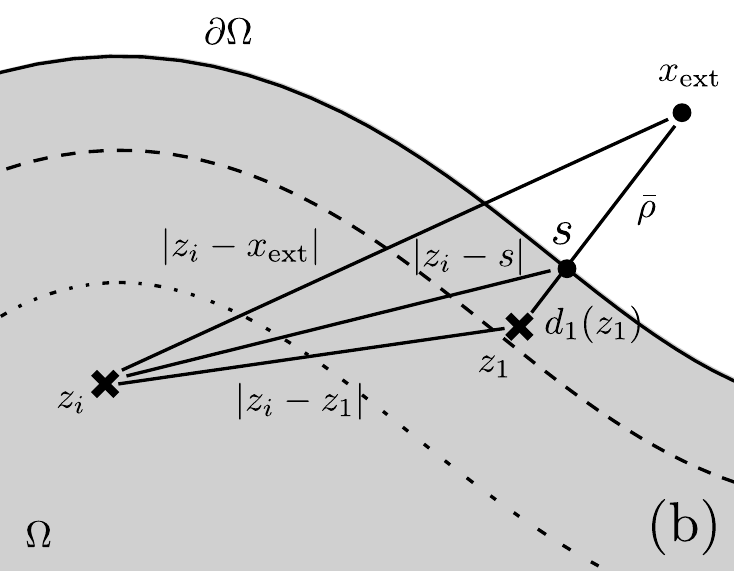}
\caption{(a) depiction of the geometric parameters $\gamma_0$ and $\delta_0$, with the dislocations $(z_1,\ldots,z_i,\ldots)$ belonging to the set $\Dom_{n,\delta_0,\gamma_0}$ defined in \eqref{eq:Geomofz}; (b) the elements involved in the estimate via triangle inequality for $|z_i-z_1|$ and \eqref{est2}.}
\label{tri}
\end{center}
\end{figure}

\subsection{Collision with the boundary}

We want to find conditions on the parameters $\delta$ and $\gamma$ in \eqref{eq:Geomofz}, in order to strengthen the constraint $\delta<\gamma$ in such a way that if the initial configuration of the system $z(0)\in\Dom_{n,\delta_0,\gamma_0}$, for some $\delta_0<\gamma_0$, then $z_1$ will collide with the boundary before any other collision event occurs.

\begin{theorem}\label{thm:boundary_collision}
Let $n\in\N$, let $\sigma\in(0,1)$, $\gamma_0>0$, and consider $\rhobar$ from \eqref{1011}.
There exist $\delta_0>0$ such that, if $z(0)\in\Dom_{n,\delta_0,\gamma_0}$, then there exists $T_{\mathrm{coll}}^{\partial\Om}>0$ such that the evolution $z(t)$ is defined for $t\in[0,T_{\mathrm{coll}}^{\partial\Om}]$, 
$z(t)\in\Omega^n$ for $t\in[0,T_{\mathrm{coll}}^{\partial\Om})$, and $z_1(T_{\mathrm{coll}}^{\partial\Om})\in\partial\Omega$ and $z'(T_{\mathrm{coll}}^{\partial\Om})\in\Omega^{n-1}$.
Furthermore, as $\delta_0\to0$, the following estimate holds
\begin{equation}\label{eq:ubht}
  T_{\mathrm{coll}}^{\partial\Om}\leq2\pi\delta_0^2+O(\delta_0^3).
  \end{equation}
\end{theorem}
The structure of the proof is the following: we first find an upper bound on the collision time for dislocation $z_1$ hitting the boundary, conditional on the configuration $z$ remaining in $\Dom_{n,\delta,\gamma}$.
In the second half of the proof, after fixing $\gamma_0\in(0,\diam\Omega/2)$, we establish a lower bound on the time at which the configuration $z$ leaves the set $\Dom_{n,\delta,\gamma_0/2}$ due to $d_{n-1}(z')$ becoming smaller than $\gamma_0/2$.
The proof is concluded by finding conditions on $\delta$ under which the former collision time is smaller than the latter: this is contained in inequality \ref{eq:time_comparison_bdry} below.
\begin{proof}
Writing the renormalised energy \eqref{renen} as in \eqref{eq:E_1rest},
the equation of motion \eqref{500} for $z_1$ reads
$$\dot z_1(t)=-\nabla_{z_1}\E_n(z_1(t),\ldots,z_n(t))=-\nabla_{z_1} \E_1(z_1(t))-\sum_{i=2}^n b_1b_i\nabla_{z_1} G_\Om(z_1(t),z_i(t)).$$
We now compute the time derivative $ \frac{\dd}{\dt}\b[\frac12d_1(z_1(t))^2\b]$ and show that it is negative, so that the dislocation $z_1(t)$ moves towards the boundary.
Indeed, we have
\begin{equation}
  \frac{\dd}{\dt}\left[\frac12d_1(z_1(t))^2\right] = A_1(z(t)),\label{eq:dynd_1}
\end{equation}
where 
$$A_1(z) :=  -d_1(z_1)\nu(s)\cdot\dot{z}_1=d_1(z_1)\nabla_{z_1}\E_1(z_1)\cdot\nu(s)+\sum_{i=2}^nb_1b_i d_1(z_1)\nabla_{z_1} G_\Om(z_1,z_i)\cdot\nu(s).$$

Now, we apply the bounds \eqref{eq:dxG_bound} and \eqref{eq:dh_bound} to estimate
\begin{align*}
  A_1(z)&\leq d_1(z_1)\nabla_{z_1}\E_1(z_i)
    \cdot\nu(s)+\sum_{i=2}^n \b|d_1(z_1)\nabla_{z_1}G_\Om(z_1,z_j)
                                  \cdot\nu(s)\b|\\
  &\leq-\frac{1}{4\pi}+\bg[\frac{C_\sigma}{\pi\rhobar}+
    \sum_{i=2}^n\frac{2\b(|z_i-x_\mathrm{ext}|^2-\rhobar^2\b)\b(\rhobar+d_1(z_1)\b)}
    {\pi\rhobar^2\b(|z_1-z_i|-d_1(z_1)\b)^2}\bg]\frac{d_1(z_1)}{2}.
\end{align*}
Fix $\gamma_0>0$, and assume that $z\in\Dom_{n,\delta,\gamma_0/2}$ for a certain $\delta\in(0,\gamma_0/4)$ (see \eqref{tri}(a)); we can estimate $A_1(z)$ and find conditions on $\delta$ in such a way that $A_1(z)<0$.
The estimates in \eqref{33} hold with $\gamma_0/2$ in place of $\gamma$ (again, see \eqref{tri}(b) for an illustration), so that, using \eqref{510}, estimate \eqref{est4} reads
\begin{equation*}
\begin{split}
\sum_{i=2}^n & \frac{2\b(|z_i-x_\mathrm{ext}|^2-\rhobar^2\b)\b(\rhobar+d_1(z_1)\b)}{\pi\rhobar^2\b(|z_1-z_i|-d_1(z_1)\b)^2} \\
& \leq\frac{2(\rhobar+\delta)(n-1)}{\pi\rhobar^2}\!\!\left(1+4\frac{\rhobar+2\delta}{\gamma_0-4\delta}+16\frac{\delta(\rhobar+\delta)}{(\gamma_0-4\delta)^2}\right).
\end{split}
\end{equation*}
In turn, we obtain
\begin{equation}\label{est5}
\begin{split}
A_1(z)\leq & -\frac{1}{4\pi}+\frac\delta{4\pi}\bg[\frac{2C_\sigma}{\rhobar}+\frac{4(\rhobar+\delta)}{\rhobar^2}(n-1)\left(1+4\frac{\rhobar+2\delta}{\gamma_0-4\delta}+16\frac{\delta(\rhobar+\delta)}{(\gamma_0-4\delta)^2}\right)\bg] \\
=& -\frac1{4\pi}(1-c(\delta)),
\end{split}
\end{equation}
where
\begin{equation}\label{eq:c}
c(\delta):=\frac{2\delta}{\rhobar}\bg[C_\sigma+\frac{2(\rhobar+\delta)}{\rhobar}(n-1)\left(1+4\frac{\rhobar+2\delta}{\gamma_0-4\delta}+16\frac{\delta(\rhobar+\delta)}{(\gamma_0-4\delta)^2}\right)\bg].
\end{equation}
Therefore, from \eqref{est5} we obtain the following estimate for \eqref{eq:dynd_1}:
\begin{equation}\label{eq:estdind_1}
 \frac{\dd}{\dt}\Big[\frac12d_1(z_1)^2\Big] = A_1(z)\leq-\frac1{4\pi}(1-c(\delta)).
\end{equation}
Expanding in powers of $\delta$ near $\delta=0$, \eqref{eq:c} can be expressed as
\begin{equation}\label{eq:cdelta}
c(\delta)=\frac{2\delta}\rhobar\left[C_\sigma+2(n-1)\left(1+4\frac\rhobar{\gamma_0}\right)\right]+O(\delta^2)=C_{\rhobar,n,\gamma_0,\sigma}\delta+O(\delta^2),
\end{equation}
which implies that there exists $\delta_0\in(0,\gamma_0/4)$ small enough such that $1-c(\delta)>0$ for all $\delta\leq\delta_0$, so that there must be a time $t=T_{\rm coll}^{\partial\Om}$ at which $z_1(T_{\rm coll}^{\partial\Om})\in\partial\Om$.

Integrating \eqref{eq:estdind_1} between $t=0$ and $t=T_\mathrm{coll}^{\partial\Om}$ entails that 
\begin{equation*}
  T_{\mathrm{coll}}^{\partial\Om}\leq\frac{2\pi d_1(z_1(0))^2}
  {1-c(\delta)}\leq\frac{2\pi \delta_0^2}
  {1-c(\delta_0)}=2\pi\delta_0^2+O(\delta_0^3),
\end{equation*}
where the second inequality holds true because the constant $C_{\rhobar,n,\gamma_0,\sigma}$ in \eqref{eq:cdelta} is strictly positive.
Estimate \eqref{eq:ubht} is proved.

Next, we note that
\begin{equation*}
  \dot{z}_i(t)= B_i(z(t)):= -\nabla_{z_i}\E_n(z(t))\quad\text{for }i=2,\ldots,n,
\end{equation*}
so using the expression of $\E_n$ in \eqref{eq:E_1rest} and the bound
\eqref{dyG_bound},
\begin{equation}\label{eq:zi_speed_bound}
\begin{split}
  |B_i(z)| \leq & \frac{1}{2\pi|z_1-z_i|}
      +\sum_{\substack{k=2\\ k\neq i}}^n \frac{1}{2\pi|z_i-z_k|}
      +\frac{n}{2\pi d_1(z_i)}\\
    \leq & \frac{1}{2\pi(d_{n-1}(z')-d_1(z_1))}
      +\frac{2n-2}{2\pi d_{n-1}(z')} \leq \frac{2n-1}{2\pi(d_{n-1}(z')-\delta_0)}.
\end{split}
\end{equation}
We want to use \eqref{eq:zi_speed_bound} to obtain bounds on the closest that any two dislocations in the ensemble $z'$ can get either to each other or to the boundary $\partial\Om$.
To do this, we bound the rate at which the minimum distance among the dislocations $z'$ (and the boundary) decreases.
Indeed, by \eqref{eq:zi_speed_bound},
\begin{equation*}
  d_{n-1}\b(z'(t)\b)-d_{n-1}\b(z'(0)\b) 
  \geq -2\int_0^t \max_{i=2,\ldots,n}|\dot{z}_i(s)|\ds
    \geq -\int_0^t\frac{2n-1}{\pi(d_{n-1}(z'(s))-\delta_0)}\ds,
\end{equation*}
and, by writing the left-hand side as $\int_0^t \dot d_{n-1}(z'(s))\ds$ we are led to solving 
\begin{equation*}
\dot d_{n-1}(z')\geq -\frac{2n-1}{\pi(d_{n-1}(z')-\delta_0)}
\end{equation*}
with initial condition $d_{n-1}(z'(0))=\gamma_0$.
We find that
\begin{equation*}
  d_{n-1}\b(z'(t)\b) \geq \delta_0+\sqrt{(\gamma_0-\delta_0)^2-\frac{4n-2}{\pi}t}.  
\end{equation*}
This entails that, given $\gamma\in(0,\gamma_0)$, $d_{n-1}(z'(t)) \geq \gamma$ for all $t\leq T(\gamma)$, with
\begin{equation}\label{eq:Tprime}
T(\gamma):=\frac{\pi}{4n-2}\Big((\gamma_0-\delta_0)^2-(\gamma-\delta_0)^2\Big).
\end{equation}
It is clear that $T(\gamma)$ is the earliest time at which either any two dislocations in $z'$ are $\gamma$ far apart from each other, or one of the dislocations in $z'$ gets to a distance $\gamma$ from the boundary $\partial\Om$.
Therefore, a sufficient condition to observe $z_1$ colliding with the boundary before any other collision event is that $T_{\mathrm{coll}}^{\partial\Om}< T(\gamma_0/2)$, i.e.,
\begin{equation}\label{eq:time_comparison_bdry}
  \frac{2\delta_0^2}{1-c(\delta_0)}<\frac{\gamma_0(3\gamma_0-4\delta_0)}{8(2n-1)}.
\end{equation}
The choice of $\delta_0$ and the fact that $4\delta_0<\gamma_0$ imply that both sides of \eqref{eq:time_comparison_bdry} are positive and, by possibly reducing $\delta_0$, it is easy to see that $\delta_0$ can be chosen such that \eqref{eq:time_comparison_bdry} is satisfied.
The theorem is proved.
\end{proof}

\begin{remark}
We note that the constant $c(\delta)$ in \eqref{eq:c} is invariant under simultaneous scaling of geometric parameters $\rhobar$, $\delta$, and $\gamma$, and is thus invariant under dilations of the coordinate system. 

Moreover, we single out two different limits for $c(\delta)$, which will be useful in the sequel: if $\Omega$ is the half space, then $\Omega$ satisfies the interior and exterior disk conditions \eqref{1011}
  with arbitrarily large $\rhobar$, so letting $\rhobar\to\infty$, we find $c(\delta)\to8\gamma_0\delta(n-1)/(\gamma_0-2\delta)^2$.
If the system consists only of one dislocation, then by plugging $n-1$ in \eqref{eq:c} we obtain $c(\delta)=2\delta C_\sigma/\rhobar$; notice that $c(\delta)=0$ for one dislocation in the half plane.

It is then clear that, focusing the attention on one particular dislocation, say $\bar z$, $c(\delta)$ tracks the influence on $\bar z$ of the other dislocations and on the geometry of the domain, in terms of the force exerted on $\bar z$.
\end{remark}

\subsection{Collision between dislocations}
We now turn to a scenario for collisions of dislocations.
We will find sufficient conditions for a collision between two dislocations to occur before any other collision event.
We suppose that we have $n\in\N$ ($n\geq2$) dislocations in $\Omega$ two of which, $z_1$ and $z_2$, with Burgers moduli $b_1=+1$ and $b_2=-1$, are much closer to each other than the others; we also suppose that the other $n-2$ dislocations, $z_3,\ldots,z_n$, are sufficiently distant from each other and from the boundary.
Our theorem states that $z_1$ and $z_2$ will collide in finite time, and that this collision happens before any other collision event occurs.

Here we adapt our notation by defining $z'':=(z_3,\ldots,z_n)$ so that a trajectory of the evolution of the configuration of the $n$ dislocations can be represented by the vector $z(t):=(z_1(t),z_2(t),z''(t))\in\Omega^n$.
In this case, the meaningful trajectories for the statement of our theorem are those that lie within sets of the form
\begin{equation*}
\begin{split}
  \mathcal{C}_{n,\zeta,\eta}:=\B\{(z_1,z_2,z'')\in\Omega^n\Bsep & |z_1-z_2|< \zeta, d_{n-2}(z'')>\eta, \\ 
  & \dist\b(\{z_1,z_2\},\{z_3,\ldots,z_n\}\cup\partial\Om\b)>\eta\B\},
   \end{split}
\end{equation*}
with $\zeta<\eta$ chosen in a way that we subsequently quantify properly.

The geometric meaning of the set $ \mathcal{C}_{n,\zeta,\eta}$ defined above is the following: if $z\in  \mathcal{C}_{n,\zeta,\eta}$, it means that $z_1$ and $z_2$ lie at a distance of at most $\zeta$ from each other, while all the other dislocations $z_3,\ldots,z_n$ lie at a distance of at least $\eta$ away from the boundary and their mutual distance is also at least $\eta$. Moreover, $z_1$ and $z_2$ are at least $\eta$ far away from any other dislocations and from the boundary.

\begin{theorem}\label{thm:collision}
Let $n\in\N$ ($n\geq2$) and let $\eta_0\in(0,\diam\Om/2)$.
There exists $\zeta_0>0$ such that, if $z(0)\in\mathcal{C}_{n,\zeta_0,\eta_0}$, then there exists $T_{\mathrm{coll}}^{\pm}>0$ such that the evolution $z(t)$ is defined for $t\in[0,T_{\mathrm{coll}}^{\pm}]$, $z(t)\in\Omega^n$ for $t\in[0,T_{\mathrm{coll}}^{\pm})$, and $z_1(T_{\mathrm{coll}}^{\pm})=z_2(T_{\mathrm{coll}}^{\pm})\in\Omega$ and $z''(T_{\mathrm{coll}}^{\pm})\in\Omega^{n-2}$.
Furthermore, as $\zeta_0\to0$, the following estimate holds
\begin{equation}\label{Tcollprime}
  T_{\mathrm{coll}}^{\pm}\leq\frac{\pi\zeta_0^2\eta_0^2}{2(\eta_0^2-\zeta_0^2-2(n-2)\zeta_0\eta_0)}.
  \end{equation}
\end{theorem}
Analogously to the proof of Theorem \ref{thm:boundary_collision}, we first find an upper bound on the collision time for dislocations $z_1$ and $z_2$, conditional on the configuration $z$ remaining in $\mathcal{C}_{n,\zeta,\eta}$.
In the second half of the proof, after fixing $\eta_0\in(0,\diam\Omega/2)$, we establish a lower bound on the time at which the configuration $z$ leaves the set $\mathcal{C}_{n,\zeta,\eta_0/2}$ due to $d_{n-2}(z'')$ becoming smaller than $\eta_0/2$.
The proof is concluded by finding conditions on $\zeta$ under which the former collision time is smaller than the latter: this is contained in inequality \eqref{928} below.
\begin{proof}
As in the proof of Theorem \ref{thm:boundary_collision}, recalling that $b_1=+1$ and $b_2=-1$, we express the renormalised energy \eqref{renen} as
\begin{equation}\label{renensplit}
  \E_n(z_1,\ldots,z_n) = \E_2(z_1,z_2) + \E_{n-2}(z'')
  +\sum_{i=3}^n b_i \b[G_\Om(z_1,z_i)-G_\Om(z_2,z_i)\b],
\end{equation}
where, again recalling \eqref{renen} with $n=2$, $b_1=+1$, and $b_2=-1$, 
\begin{equation}\label{renen2}
\E_2(z_1,z_2)=\frac1{2\pi}\log|z_1-z_2|+\frac12 h_\Om(z_1)+\frac12 h_\Om(z_2)- k_\Om(z_1,z_2).
\end{equation}
The equations of motion for $z_1$ and $z_2$ read
\begin{equation}\label{eq_mot1}
\begin{split}
\dot z_1(t)= & -\nabla_{z_1} \E_n(z(t))=-\nabla_{z_1} \E_2(z_1(t),z_2(t))-\sum_{i=3}^n b_i\nabla_{z_1}G_\Om(z_1(t),z_i(t)) \\
=& \frac1{2\pi}\frac{z_2(t)-z_1(t)}{|z_1(t)-z_2(t)|^2}-\frac12\nabla_{z_1} h_\Om(z_1(t))+\nabla_{z_1}k_\Om(z_1(t),z_2(t)) \\
& -\sum_{i=3}^n b_i\nabla_{z_1}G_\Om(z_1(t),z_i(t))
\end{split}
\end{equation}
\begin{equation}\label{eq_mot2}
\begin{split}
\dot z_2(t)= & -\nabla_{z_2} \E_n(z(t))=-\nabla_{z_2} \E_2(z_1(t),z_2(t))+\sum_{i=3}^n b_i\nabla_{z_2}G_\Om(z_2(t),z_i(t)) \\
=& \frac1{2\pi}\frac{z_1(t)-z_2(t)}{|z_1(t)-z_2(t)|^2}-\frac12\nabla_{z_2} h_\Om(z_2(t))+\nabla_{z_2}k_\Om(z_1(t),z_2(t)) \\
& +\sum_{i=3}^n b_i\nabla_{z_2}G_\Om(z_2(t),z_i(t)) \\
\end{split}
\end{equation}

We now fix $\eta_0>0$, consider $\zeta\in(0,\eta_0/2)$, and for $z\in\mathcal{C}_{n,\zeta,\eta_0/2}$ compute the time derivative 
\begin{equation*}
  \frac{\dd}{\dt}\left[\frac12|z_1(t)-z_2(t)|^2\right] = A_2(z(t)),
\end{equation*}
where $A_2(z) := (z_1-z_2)\cdot (\dot z_1(t)-\dot z_2(t))$.
We will determine conditions on $\zeta$ in such a way that $A_2(z)<0$ and therefore the dislocations $z_1$ and $z_2$ attract.

Using the explicit expressions from \eqref{eq_mot1} and \eqref{eq_mot2}, the contribution to $A_2(z)$ coming from $\E_2(z_1,z_2)$ is
\begin{equation}\label{firstpart}
\begin{split}
(z_1-z_2) \cdot\bigg(  -\frac{z_1-z_2}{\pi|z_1-z_2|^2} & -\frac12\nabla_{z_1} h_\Om(z_1)+\nabla_{z_1}k_\Om(z_1,z_2) \\ 
&\qquad\qquad+\frac12\nabla_{z_2} h_\Om(z_2)-\nabla_{z_2}k_\Om(z_1,z_2)\bigg), \\
= -\frac1{\pi}+(z_1-z_2)\cdot\bigg( & \nabla_{z_1}k_\Om(z_1,z_2)  -\frac12\nabla_{z_1} h_\Om(z_1) \\ 
&\qquad\qquad-\nabla_{z_2}k_\Om(z_1,z_2)+\frac12\nabla_{z_2} h_\Om(z_2)\bigg).
\end{split}
\end{equation}
Next, by Taylor expanding the function $z_2\mapsto F_{z_1}(z_2):= \nabla_{z_1}k_\Om(z_1,z_2)$ about the point $z_1$ using the Lagrange form for the remainder, and recalling \eqref{eq:h_Om}, we obtain 
\begin{equation}\label{Taylor1}
  \left|(z_1-z_2)\cdot\left(\nabla_{z_1}k_\Om(z_1,z_2)-\frac12\nabla_{z_1}h_\Om(z_1)\right)\right| =  |(z_1-z_2)\cdot \nabla_{z_2}\nabla_{z_1}k_\Om(z_1,\theta)(z_2-z_1)|,
\end{equation}
where $\theta$ is an intermediate point in the line segment joining $z_1$ and $z_2$.
Notice now that $z_2\mapsto F_{z_1}(z_2)$ is a vector-valued harmonic function which, by \eqref{eq:k_Om_problem}, satisfies 
\begin{equation*}
\begin{cases}
-\Delta F_{z_1}(z_2)=0 & \text{in $\Omega$,} \\
F_{z_1}(z_2)=\displaystyle\frac{z_2-z_1}{2\pi|z_1-z_2|^2} & \text{on $\partial\Omega$.}
\end{cases}
\end{equation*}
Applying \eqref{eq:gradient_bound} to $F_{z_1}(z_2)$ on the ball centred at $\theta$ of radius $r=\dist(\theta,\partial\Om)$, and using the maximum principle on the domains $B_r(\theta)\subset\Omega$, we can estimate \eqref{Taylor1} by
\begin{equation}\label{Taylor2}
\begin{split}
|z_1-z_2|^2 \cdot& |\nabla_{z_2}\nabla_{z_1} k_\Om(z_1,\theta)| \leq |z_1-z_2|^2 \frac2{\dist(\theta,\partial\Omega)}\sup_{y\in\partial B_r(\theta)} |F_{z_1}(y)| \\
\leq & \frac{2|z_1-z_2|^2}{\dist(\theta,\partial\Omega)}\sup_{y\in\partial \Omega} |F_{z_1}(y)|= \frac{|z_1-z_2|^2}{\dist(\theta,\partial\Omega)}\sup_{y\in\partial \Omega} \frac{1}{\pi|z_1-y|} \\
\leq & \frac{|z_1-z_2|^2}{\pi\dist(\theta,\partial\Omega)\dist(z_1,\partial\Omega)}\leq 4\frac{|z_1-z_2|^2}{\pi \eta_0^2}.
\end{split}
\end{equation}
From \eqref{Taylor1} and \eqref{Taylor2} we have obtained that 
\begin{equation}\label{Taylor3}
\left|(z_1-z_2)\cdot\left(\nabla_{z_1}k_\Om(z_1,z_2)-\frac12\nabla_{z_1}h_\Om(z_1)\right)\right|\leq 4\frac{|z_1-z_2|^2}{\pi \eta_0^2}
\end{equation}
Applying the same argument to the function $z_1\mapsto F_{z_2}(z_1):=\nabla_{z_2}k_\Om(z_1,z_2)$ and recalling that $z\in\mathcal{C}_{n,\zeta,\eta_0/2}$, we obtain
\begin{equation}\label{Taylor4}
\left|(z_1-z_2)\cdot\left(\nabla_{z_2}k_\Om(z_1,z_2)-\frac12\nabla_{z_2}h_\Om(z_2)\right)\right|\leq 4\frac{|z_1-z_2|^2}{\pi \eta_0^2}
\end{equation}
By using \eqref{Taylor3} and \eqref{Taylor4}, and the fact that $|z_1-z_2|<\zeta$, we can bound \eqref{firstpart} by
\begin{equation}\label{firstest}
\begin{split}
-\frac1{\pi}+ & (z_1-z_2)\cdot\bigg(\nabla_{z_1}k_\Om(z_1,z_2)-\frac12\nabla_{z_1} h_\Om(z_1)-\nabla_{z_2}k_\Om(z_1,z_2)+\frac12\nabla_{z_2} h_\Om(z_2)\bigg) \\
\leq & -\frac1{\pi}+\frac{8\zeta^2}{\pi \eta_0^2}.
\end{split}
\end{equation}
We turn now to the remaining interaction terms in the estimate of $A_2(z)$, namely, we are left with estimating
$$\bigg|(z_1-z_2)\cdot\bigg(-\sum_{i=3}^n b_i\nabla_{z_1}G_\Om(z_1(t),z_i(t))-\sum_{i=3}^n b_i\nabla_{z_2}G_\Om(z_2(t),z_i(t))\bigg)\bigg|.$$
By using the estimate \eqref{dyG_bound} on the gradient of the Green's function $G_\Om$, we obtain
\begin{equation}\label{secondest}
\begin{split}
\bigg|(z_1-z_2)\cdot & \bigg(-\sum_{i=3}^n b_i\nabla_{z_1}G_\Om(z_1(t),z_i(t))-\sum_{i=3}^n b_i\nabla_{z_2}G_\Om(z_2(t),z_i(t))\bigg)\bigg| \\
\leq & |z_1-z_2|\sum_{i=3}^n\b(|\nabla_{z_1}G_\Om(z_1(t),z_i(t))|+|\nabla_{z_2}G_\Om(z_2(t),z_i(t))|\b) \\
\leq & |z_1-z_2|\sum_{i=3}^n\left(\frac1{2\pi|z_1-z_i|}+\frac1{2\pi d_1(z_1)}+\frac1{2\pi|z_2-z_i|}+\frac1{2\pi d_1(z_2)}\right) \\
\leq & \frac{4(n-2)\zeta}{\pi\eta_0}.
\end{split}
\end{equation}
Combining now \eqref{firstest} and \eqref{secondest}, we finally obtain
\begin{equation}\label{thirdest}
 A_2(z)\leq -\frac1{\pi}\left(1 - \frac{8\zeta^2}{\eta_0^2}-\frac{4(n-2)\zeta}{\eta_0}\right)=:-\frac1\pi(1-c(\zeta)).
\end{equation}
It is easy to see that $c(\zeta)\geq0$ and that it is smaller than $1$ for 
$$\zeta<\zeta_0:=\eta_0\b(\sqrt{(n-2)^2+2}-(n-2)\b)/4,$$ 
so that $A_2(z)<0$ if $z\in\mathcal{C}_{n,\zeta,\eta_0/2}$.

Integrating \eqref{thirdest} between $t=0$ and $t=T_{\mathrm{coll}}^{\pm}$ (for which $z_1(T_{\mathrm{coll}}^{\pm})=z_2(T_{\mathrm{coll}}^{\pm})$), we obtain
\begin{equation*}
T_{\mathrm{coll}}^{\pm}\leq\frac{\pi|z_1(0)-z_2(0)|^2}{2(1-c(\zeta))}\leq\frac{\pi\zeta_0^2}{2(1-c(\zeta_0))}=\frac{\pi\zeta_0^2\eta_0^2}{2(\eta_0^2-8\zeta_0^2-4(n-2)\zeta_0\eta_0)},
\end{equation*}
where we have used that $z\in\mathcal{C}_{n,\zeta_0,\eta_0/2}$ and the monotonicity of $c(\zeta)$ for $\zeta>0$.
Estimate \eqref{Tcollprime} is proved.

For $z\in\mathcal{C}_{n,\zeta,\eta_0/2}$ and $i=3,\ldots,n$, we have that 
\begin{equation*}
\dot z_i(t)=B_i(z(t)):=-\nabla_{z_i}\E_n(z(t)).
\end{equation*}
By using the expression of $\E_n$ in \eqref{renensplit}, we have to estimate
\begin{equation}\label{est10}
\begin{split}
|B_i(z)|\leq & \sum_{\substack{j=3\\j\neq i}}^n |\nabla_{z_i} G_\Om(z_i,z_i)|+\frac12|\nabla_{z_i} h_\Om(z_i)|+|\nabla_{z_i} G_\Om(z_i,z_1)-\nabla_{z_i} G_\Om(z_i,z_2)| \\ 
=: & B_i^1+B_i^2+B_i^3.
\end{split}
\end{equation}
We notice that we separate the contributions of $z_1$ and $z_2$ because finer estimates are available for them since $|z_1-z_2|<\zeta$.
We estimate the three contributions separately.
Using \eqref{dyG_bound}, the fact that $i,j\notin\{1,2\}$, we have 
$$|\nabla_{z_i}G_\Om(z_i,z_j)|\leq\frac{1}{2\pi|z_i-z_j|}+\frac1{2\pi d_1(z_i)}\leq\frac1{\pi d_{n-2}(z'')},$$
for each $i,j\in\{3,\ldots,n\}$, and $j\neq i$, which implies that
\begin{equation}\label{921}
B_i^1\leq\frac{n-3}{\pi d_{n-2}(z'')},\qquad\text{for $i\in\{3,\ldots,n\}$}.
\end{equation}
We use \eqref{esthfar} to estimate $B_i^2$.
By \eqref{estnablah} we can bound
\begin{equation}\label{922}
B_i^2\leq\frac{\max\{-\log d_1(z_i),\lo\}}{\pi d_1(z_i)}, \qquad\text{for $i\in\{3,\ldots,n\}$}.
\end{equation}
We are left with the estimate for $B_i^3$.
We will apply the mean value theorem to the function $x\mapsto H_i(x):=\nabla_{z_i}G_\Om(x,z_i)$.
Let $\theta$ belong to the line segment joining $z_1$ and $z_2$, and recall the definition \eqref{eq:Green} of $G_\Om$. Then
\begin{equation*}
\begin{split}
B_i^3= & |H_i(z_1)-H_i(z_2)|=|\nabla_xH_i(\theta)(z_1-x_2)| \\
= & \!\left|-\frac1{2\pi}(\nabla_x\nabla_{z_i} \log|\theta-z_i|)(z_1-z_2)+(\nabla_x\nabla_{z_i}k_\Om(\theta,z_i))(z_1-z_2)\right| \\
\leq & \!\left|\frac1{2\pi|\theta-z_i|^2}\left(\mathbb{I}-2\frac{\theta-z_i}{|\theta-z_i|}\otimes\frac{\theta-z_i}{|\theta-z_i|}\right)(z_1-z_2)\right|+|(\nabla_x\nabla_{z_i}k_\Om(\theta,z_i))(z_1-z_2)|;
\end{split}
\end{equation*}
here, $\mathbb{I}\in\R^{2\times2}$ denotes the identity matrix.
We notice that matrices of the type $\mathbb{I}-2e\otimes e$, where $e$ is a unit vector, are reflections, and their operator norm is $1$, so that, recalling the argument used in \eqref{Taylor2} and that $z\in\mathcal{C}_{n,\zeta,\eta_0/2}$,
\begin{equation}\label{923}
\begin{split}
B_i^3\leq & \frac{\zeta}{2\pi|\theta-z_i|^2}+\frac{\zeta}{\pi \dist(\theta,\partial\Omega)d_{n-2}(z'')} \\
\end{split}
\end{equation}
Putting \eqref{921}, \eqref{922}, and \eqref{923} together, and noting that $d_{n-2}(z'')\leq d_1(z_i)$, \eqref{est10} finally becomes
\begin{equation*}
\begin{split}
|B_i(z)|\leq & \frac{n-3+\max\{-\log d_{n-2}(z''),\lo\}}{\pi d_{n-2}(z'')}+\frac{\zeta}{2\pi|\theta-z_i|^2} \\
& + \frac{\zeta}{\pi \dist(\theta,\partial\Omega)d_{n-2}(z'')}.
\end{split}
\end{equation*}
Considering the ensemble $z_\theta:=(\theta,z'')=(\theta,z_3,\ldots,z_n)\in\Omega^{n-1}$, by definition of $d_{n-1}$ (see \eqref{1001}), we have that $d_{n-1}(z_\theta)\leq\min\{d_{n-2}(z''),\dist(\theta,\partial\Om),|\theta-z_i|\}$, so that
\begin{equation}\label{925}
\begin{split}
|B_i(z)|\leq & \frac{n-3+\max\{-\log d_{n-1}(z_\theta),\lo\}}{\pi d_{n-1}(z_\theta)}+\frac{3\zeta}{2\pi d_{n-1}^2(z_\theta)} \\
\leq & \frac{n-3+\lo+d_{n-1}^{-1}(z_\theta)}{\pi d_{n-1}(z_\theta)}+\frac{3\zeta}{2\pi d_{n-1}^2(z_\theta)} \\
= & \frac{2(n-3+\lo)d_{n-1}(z_\theta)+2+3\zeta}{2\pi d_{n-1}^2(z_\theta)}=\frac{\Lo d_{n-1}(z_\theta)+2+3\zeta}{2\pi d_{n-1}^2(z_\theta)},
\end{split}
\end{equation}
where we have used that $\max\{-\log d_{n-1}(z_\theta),\lo\}\leq\lo+d_{n-1}^{-1}(z_\theta)$ and we have defined $\Lo:=2(n-3+\lo)$.

As in the proof of Theorem \ref{thm:boundary_collision}, we want to use \eqref{925} to obtain bounds on $d_{n-1}(z_\theta)$, that is, we want to control how close any of the dislocations in $z''$ approach one another or the boundary, and additionally prevent them from getting too close to $z_1$ and $z_2$.
To do this, we bound the rate at which $d_{n-1}(z_\theta)$ decreases.
Indeed, by \eqref{925},
$$d_{n-1}(z_\theta(t))-d_{n-1}(z_\theta(0))\geq -2\int_0^t \max_{i=2,\ldots,n}|\dot{z}_i(s)|\ds\geq -\int_0^t \frac{\Lo d_{n-1}(z_\theta)+2+3\zeta}{\pi d_{n-1}^2(z_\theta)}\ds,$$
and, by writing the left-hand side as $\int_0^t \dot d_{n-1}(z_\theta(s))\ds$ we are led to solving 
\begin{equation*}
\dot d_{n-1}(z_\theta)\geq -\frac{\Lo d_{n-1}(z_\theta)+2+3\zeta}{\pi d_{n-1}^2(z_\theta)}
\end{equation*}
with initial condition $d_{n-1}(z_\theta(0))=\eta_0$.
Noting that the corresponding equation is separable, and defining $\chi=\chi(\zeta):=2+3\zeta$, we find that the time evolution $t\mapsto d_{n-1}(z_\theta(t))$ satisfies
\begin{equation*}
\begin{split}
t\leq & \frac\pi{\Lo^2}\left(\chi(d_{n-1}(z_\theta(t))-\eta_0)-\frac\Lo2(d_{n-1}^2(z_\theta(t))-\eta_0^2)\right)\\
& +\frac{\pi\chi^2}{\Lo^3} \log\left(\frac{\Lo\eta_0+\chi}{\Lo d_{n-1}(z_\theta(t))+\chi}\right).
\end{split}
\end{equation*}
This entails that, given $\eta\in(0,\eta_0)$, $d_{n-1}(z_\theta(t)) \geq \eta$ for all $t\leq T(\eta)$, with
\begin{equation}\label{927}
T(\eta):=\frac\pi{\Lo^2}\left(\chi(\eta-\eta_0)-\frac\Lo2(\eta^2-\eta_0^2)+\frac{\chi^2}{\Lo} \log\left(\frac{\Lo\eta_0+\chi}{\Lo\eta+\chi}\right)\right).
\end{equation}
As for the time $T(\gamma)$ in \eqref{eq:Tprime}, the time $T(\eta)$ above is the earliest time at which any two points in $z_\theta$ are $\eta$ apart from each other or from the boundary.
Therefore, a sufficient condition to observe $z_1$ and $z_2$ colliding before any other collision event is that $T_{\mathrm{coll}}^{\pm}< T(\eta_0/2)$, that is, recalling \eqref{Tcollprime} and \eqref{927}, 
\begin{equation}\label{928}
\frac{\pi\zeta_0^2}{2(1-c(\zeta_0))}<\frac\pi{\Lo^2}\left(\frac{3\Lo\eta_0^2}8-\frac{\chi(\zeta_0)\eta_0}2+\frac{\chi(\zeta_0)^2}{\Lo} \log\left(2\frac{\Lo\eta_0+\chi(\zeta_0)}{\Lo\eta_0+2\chi(\zeta_0)}\right)\right).
\end{equation}
Recalling that $\chi(\zeta_0)=2+3\zeta_0$, it is clear that the right-hand side above is of order $1$ as $\zeta_0\to0$, whereas the left-hand side is of order $\zeta_0^2$. 
Since both sides in \eqref{928} are positive, by possibly reducing it, $\zeta_0$ can be chosen such that \eqref{928} is satisfied.
The theorem is proved.
\end{proof}

\begin{remark}
Contrary to Theorems \ref{thm:fatal} and \ref{thm:boundary_collision}, Theorem \ref{thm:collision} does not require any regularity on the boundary $\partial\Om$, since bounding the forces due to the boundary and other dislocations relies only upon estimates \eqref{dyG_bound} and \eqref{estnablah}, which are independent of the assumption that $\partial\Omega$ is $\CC^2$.
\end{remark}

\section{Comparison of collision bounds - Explicit examples and numerics}\label{sect:collision}
In this section, we compare our analytical results with exactly solvable cases (see also the discussion in Section 3 of \cite{BFLM15}) and numerical simulations.
More precisely, we consider domains that are the unit disk, the half plane, and the plane.
Although the theorems and lemmas from the previous sections apply only to bounded domains, a careful inspection of the renormalised energy \eqref{renen} shows that the evolution \eqref{500} is also well defined in the case of unbounded domains.
The exactly solvable cases that we consider are: one dislocation $z$ in (i) the half plane and (ii) the unit disk, (iii) two dislocations $z_1$, $z_2$ of opposite Burgers modulus $b_1=+1=-b_2$ in the unit disk, and (iv) two dislocations in the plane.
In cases (i) and (ii) the dislocation hits the boundary in finite time; case (iii) shows a variety of scenarios, \emph{i.e.}, collisions with the boundary, collision between dislocations, and unstable equilibria; in case (iv) the dislocations will collide, or the evolution exists for all time.
In cases (i) and (ii), the renormalised energy \eqref{renen} reads, recalling \eqref{eq:h_Om},
\begin{equation}\label{951}
\E_1(z)=\frac12 h_\Om(z)=\frac12 k_\Om(z,z).
\end{equation}
In case (iii), the energy takes the form \eqref{renen2}, whereas in case (iv) it reduces to
\begin{equation}\label{952}
\E_2(z_1,z_2)=-\frac{b_1b_2}{2\pi} \log|z_1-z_2|.
\end{equation}

To numerically validate our conclusions, we collected data on the time required for dislocations with random initial conditions in the circle to hit the boundary. We also plotted trajectories for a single dislocation in domains whose Green's functions is not known explicitly, namely a square and a cardioid, which show agreement with the conclusion of Theorem \ref{thm:fatal}.

\subsection{One dislocation in the half plane}
Let $n=1$ and let $z=(x_1,x_2)\in\Om:=\R^2_+$, the upper half plane.  
Notice that $\partial\Omega=\{(x_1,0)\in\R^2\}$.
Since, given $u,v\in\Om$, the Green's function is $G_\Om(u,v)=-\frac1{2\pi}\log|u-v|+\frac1{2\pi}\log|\bar u-v|$, with $\bar u=(u_1,-u_2)$, and since from \eqref{eq:Green} $k_\Om(u,v)=\frac1{2\pi}\log|\bar u-v|$, from \eqref{eq:h_Om} we have $h_\Om(u)=\frac1{2\pi}\log|\bar u-u|$, the renormalised energy \eqref{951} for one dislocation in the half plane reads
\begin{equation}\label{931}
\E_1(z)=\frac12 h_\Om(z)=\frac{1}{8\pi}\log(2x_2^2),
\end{equation}
from which it emerges that $\E_1$ is translation-invariant with respect to the coordinate $x_1$, so that we consider $z=(0,x_2)$.
The equation of motion for $z$ is now determined by
\begin{equation}\label{932}
\dot z(t)=\left(
\begin{array}{c}
\dot x_1(t) \\
\dot x_2(t)
\end{array}\right)=-\frac12\nabla_z h_\Om(z(t))=-\frac1{4\pi x_2(t)}
\left(
\begin{array}{c}
0\\
1
\end{array}\right),
\end{equation}
with initial condition $z(0)=(0,\delta)$.
Solving the dynamics \eqref{932}, yields the existence of $T_{\mathrm{coll}}^{\partial\Om}=2\pi\delta^2$ such that $x_1(t)=0$ and $x_2(t)=\sqrt{\delta^2-t/2\pi}$, for all $t\in[0,T_{\mathrm{coll}}^{\partial\Om}]$.
The collision time $T_{\mathrm{coll}}^{\partial\Om}$ is characterised as that time for which $x_2(T_{\mathrm{coll}}^{\partial\Om})=0$.

For $n=1$ and letting $\rhobar\to+\infty$, the dislocation $z$ is in $\Dom_{1,\delta_0,\infty}$ for every $\delta_0>\delta$.
By Theorem \ref{thm:boundary_collision}, the upper bound on the collision time (see \eqref{eq:ubht}) becomes $T_{\mathrm{coll}}^{\partial\Om}\leq 2\pi\delta^2$, since in this situation the constant $c(\delta)$ in \eqref{eq:c} vanishes.
This agrees exactly with the collision time obtained solving \eqref{932}.

\subsection{One dislocation in the unit disk}
Let $n=1$ and let $z=(x_1,x_2)\in\Om:=B_1(0)$, the unit disk, be such that $z\neq0$.
Setting $\delta=d_1(z)=1-|z|$, and evaluating \eqref{eq:h_interior_ball} for $\rhobar=1$, the energy \eqref{931} reads
\begin{equation*}
\E_1(z)=\frac12 h_\Om(z)=\frac1{4\pi}\log(2\delta-\delta^2).
\end{equation*}
The equation of motion $\dot z(t)=-\nabla_z \E_1(z(t))$ takes the form
\begin{equation}\label{934}
\dot z(t)=\frac{z}{2\pi(1-|z|^2)},
\end{equation}
which, by taking the scalar product with $z$ and setting $R(t):=|z(t)|^2$, yields the implicit equation
\begin{equation}\label{935}
\log \frac{R(t)}{(1-\delta)^2}-R(t)+(1-\delta)^2=\frac{t}\pi,
\end{equation}
where we have used $R(0)=|z(0)|^2=(1-\delta)^2$.
Solving \eqref{935} for $T_{\mathrm{coll}}^{\partial\Om}$ for which $R(T_{\mathrm{coll}}^{\partial\Om})=1$, gives
\begin{equation}\label{936}
T_{\mathrm{coll}}^{\partial\Om}=\pi(\delta^2-2\delta-2\log(1-\delta))=2\pi\delta^2+O(\delta^3),
\end{equation}
as $\delta\to0$.

In this case, the dislocation $z$ is in $\Dom_{1,\delta_0,1}$, for every $\delta_0\in(\delta,1)$ and $c(\delta)$ in \eqref{eq:c} becomes $c(\delta)=2\delta C_\sigma$, where $C_\sigma$ is given in \eqref{Csigma}.
As we can choose $\sigma=\delta$, we have $c(\delta)=O(\delta)$, so that the upper bound \eqref{eq:ubht} on the collision time obtained in Theorem \ref{thm:boundary_collision} becomes $T_{\mathrm{coll}}^{\partial\Om}\leq2\pi\delta^2+O(\delta^3)$, which again agrees exactly with the collision time obtained solving \eqref{934}.

Notice that if the dislocation is located at the centre of the disk, that is $z=0$, then $\E_1(z)=0$ and also the velocity $\dot z(t)$ in \eqref{934} vanishes, so that the dislocation will not move.
Consistently, the exact form of the collision time $T_{\mathrm{coll}}^{\partial\Om}$ in \eqref{936} tends to $+\infty$ as $\delta\to1$.
Similarly, the upper bound obtained in Theorem \ref{thm:boundary_collision} tends to $+\infty$ (noting how $c_\sigma$ in \eqref{913} depends on $\sigma$, it is clear that $c_\sigma\to+\infty$ as $\sigma\to1$).

\subsection{Two dislocations in the unit disk}\label{twodisk}
Let $n=2$, let $z_1,z_2\in\Om:=B_1(0)$, and let the Burgers moduli be $b_1=+1$, $b_2=-1$.
Taking into account \eqref{eq:Green} and \eqref{eq:Green_func_ball_interior}, the function $k_\Om(z_1,z_2)$ has the expression
\begin{equation*}
k_\Om(z_1,z_2)=\frac1{4\pi}\log(1-2z_1\cdot z_2+|z_1|^2|z_2|^2),
\end{equation*}
so that, using \eqref{eq:h_Om}, the renormalised energy \eqref{renen2} reads
\begin{equation}\label{938}
\begin{split}
\E_2(z_1,z_2)= & \frac1{2\pi}\log|z_1-z_2|+\frac1{4\pi}\log(1-|z_1|^2)+\frac1{4\pi}\log(1-|z_2|^2) \\
& -\frac1{4\pi}\log(1-2z_1\cdot z_2+|z_1|^2|z_2|^2).
\end{split}
\end{equation}
It is immediate to see that \eqref{938} is invariant under rotations in the plane, so that solving the equations of motion \eqref{eq_mot1} and \eqref{eq_mot2} is equivalent to solving the following equations for the radial coordinates $r_1(t):=|z_1(t)|$ and $r_2(t):=|z_2(t)|$, and for the angle $\varphi(t)$ formed by $z_1(t)$ and $z_2(t)$ (satisfying $z_1(t)\cdot z_2(t)=r_1(t)r_2(t)\cos\varphi(t)$).
Plugging the definitions of $r_1(t)$, $r_2(t)$, and $\varphi(t)$ in \eqref{eq_mot1} and \eqref{eq_mot2} (and suppressing the time dependence for the sake of readability), we obtain
\begin{equation}\label{939}
\begin{split}
\dot r_1=  \frac1{2\pi}\bigg[ & \frac{r_2\cos\varphi-r_1}{r_1^2+r_2^2-2r_1r_2\cos\varphi}+\frac{r_1}{1-r_1^2} -\frac{r_2\cos\varphi-r_1r_2^2}{1-2r_1r_2\cos\varphi+r_1^2r_2^2}\bigg], \\
\dot r_2=  \frac1{2\pi}\bigg[ & \frac{r_1\cos\varphi-r_2}{r_1^2+r_2^2-2r_1r_2\cos\varphi}+\frac{r_2}{1-r_2^2} -\frac{r_1\cos\varphi-r_1^2r_2}{1-2r_1r_2\cos\varphi+r_1^2r_2^2}\bigg], \\
\dot\varphi = \frac{1}{2\pi} \bigg[ & \frac{(r_1^2+r_2^2)(r_1^2+r_2^2-r_1^2r_2^2-1)\sin\varphi} {r_1r_2(1-2r_1r_2\cos\varphi+r_1^2r_2^2)(r_1^2+r_2^2-2r_1r_2\cos\varphi)}\bigg].
\end{split}
\end{equation}
We are going to discuss particular scenarios.
Let us consider the initial condition with $z_1(0)$ and $z_2(0)$ aligned on a diameter on opposite sides of the centre, that is, when $\varphi(0)=\pi$.
The equation for $\varphi$ in \eqref{939} reduces to $\dot\varphi(t)=0$ for all $t>0$, that is, the dislocations keep the alignment.
The first two equations in \eqref{939} then reduce to
\begin{equation*}
\dot r_1=  \frac1{2\pi}\bigg[  -\frac{1}{r_1+r_2}+\frac{r_1}{1-r_1^2} +\frac{r_2}{1+r_1r_2}\bigg], \;
\dot r_2=  \frac1{2\pi}\bigg[  -\frac{1}{r_1+r_2}+\frac{r_2}{1-r_2^2} +\frac{r_1}{1+r_1r_2}\bigg].
\end{equation*}
If $z_1(0)=-z_2(0)$, since $r_1(0)=r_2(0)=:r_0$, the right-hand sides of the equations above are the same, so that $r_1(t)$ and $r_2(t)$ evolve with the same law, namely, calling $r(t):=r_1(t)=r_2(t)$,
\begin{equation}\label{940}
\dot r(t)=  \frac{r^4(t)+4r^2(t)-1}{4\pi r(t)(1-r^4(t))},\qquad\text{with $r(0)=r_0$.}
\end{equation}
We can now study the sign of $\dot r(0)$ to find if the dislocation will collide with each other or with the boundary, or if the are in equilibrium.
We obtain that if $r_0=\sqrt{\sqrt5-2}$ then the dislocations are in equilibrium and do not move, since for this value of the initial condition, all the right-hand sides of \eqref{939} vanish.
If $r_0<\sqrt{\sqrt5-2}$, the dislocations will collide with each other at the centre of the disk (by symmetry), whereas if $r_0>\sqrt{\sqrt5-2}$, the dislocations will collide with the boundary.

We can now compute the collision time in the last two scenarios by integrating \eqref{940}.
It is convenient to set $R(t):=r^2(t)$, so that it solves
\begin{equation*}
\dot R(t)=  \frac{R^2(t)+4R(t)-1}{2\pi (1-R^2(t))},\qquad\text{with $R(0)=R_0=r_0^2$.}
\end{equation*}
We obtain
$$
R(t)-R_0+\bigg(\frac{4\sqrt5}5-2\bigg)\log\frac{R(t)+2-\sqrt5}{R_0+2-\sqrt5}-\bigg(\frac{4\sqrt5}5+2\bigg)\log\frac{R(t)+2+\sqrt5}{R_0+2+\sqrt5}=-\frac{t}{2\pi},
$$
from which we find the times of collision with the boundary, by imposing $R(T_{\mathrm{coll}}^{\partial\Om})=1$, and of $z_1$ with $z_2$, by imposing $R(T_{\mathrm{coll}}^{\pm})=0$.
The computations give
\begin{subequations}
\begin{equation}\label{943a}
T_{\mathrm{coll}}^{\partial\Om}=  2\pi\bigg[r_0^2-1-\frac{4\sqrt5}5\log\frac{(7-3\sqrt5)(r_0^2+2+\sqrt5)}{2(r_0^2+2-\sqrt5)}-2\log\frac{r_0^4+4r_0^2-1}{4}\bigg],
\end{equation}
\begin{equation}\label{943b}
T_{\mathrm{coll}}^{\pm}= 2\pi\bigg[r_0^2+\frac{4\sqrt5}5\log\bigg(\frac{(9+4\sqrt5)(r_0^2+2-\sqrt5)}{-(r_0^2+2+\sqrt5)}\bigg)-2\log(1-4r_0^2-r_0^4)\bigg].
\end{equation}
\end{subequations}

When $r_0\to1$ in \eqref{943a}, setting $\delta=1-r_0$, asymptotically $T_{\mathrm{coll}}^{\partial\Om}$ behaves like $2\pi\delta^2$, which shows agreement with the upper bound found in \eqref{eq:ubht}.

When $r_0\to0$ in \eqref{943b}, setting $\zeta=2r_0$, asymptotically $T_{\mathrm{coll}}^{\pm}$ behaves like $\pi\zeta^2/2$, which shows agreement with the upper bound found in \eqref{Tcollprime}.

\subsection{Two dislocations in the plane}
Let $n=2$, let $z_1,z_2\in\Omega:=\R^2$, and let $b_1$ and $b_2$ be their Burgers moduli.
Recalling \eqref{952}, it is easy to see that the renormalised energy is translation invariant, so that we can choose $z_2(0)=-z_1(0)$.
The equations of motion read
\begin{equation}\label{944}
\dot z_1(t)=\frac{b_1b_2}{2\pi}\frac{z_1(t)-z_2(t)}{|z_1(t)-z_2(t)|^2}=-\dot z_2(t),
\end{equation}
so that the barycentre $z(t):=\frac12(z_1(t)+z_2(t))$, satisfies $\dot z(t)=0$ for all times.
Since $z(0)=0$, the dislocations are always at a symmetric position across the centre of the coordinate system.
Therefore, \eqref{944} becomes
\begin{equation}\label{945}
\dot z_1(t)=\frac{b_1b_2}{4\pi}\frac{z_1(t)}{|z_1(t)|^2}=-\dot z_2(t).
\end{equation}
Integrating \eqref{945}, one obtains
$$z_1(t)=z_1(0)\sqrt{1+\frac{b_1b_2t}{2\pi|z_1(0)|^2}}=-z_2(t),$$
from which it is clear that if the dislocations have equal Burgers moduli, that is if $b_1b_1=1$, then the evolution exists for all times $t\geq0$ and $z_1$ and $z_2$ grow arbitrarily far apart.
If, on the other side, $b_1b_1=-1$, then the dislocations attract and the evolution exists up to the collision time $2\pi|z_1(0)|^2$, which again is in agreement with the estimates on $T_{\mathrm{coll}}^{\pm}$ in \eqref{Tcollprime}, again recalling that $\zeta=2|z_1(0)|$.

\subsection{Plots from numerical simulations}
Here, we include plots from numerical simulations of the dynamics in different scenarios. 
All calculations are performed in Matlab R$2016$b and the trajectories were integrated using a built--in stiff solver.

Figure \ref{yeah} shows the superposition of 5000 runs of the scenario described in Subsection \ref{twodisk}, where initial conditions have been randomly generated.
In all of the runs in the unit disk ($\rhobar=1$), we have chosen $\delta_0=0.2$ and $\gamma_0=0.5$, so that the initial condition $(z_1(0),z_2(0))\in\Dom_{1,0.2,0.5}$ (see \eqref{eq:Geomofz}); we have also chosen Burgers moduli $b_1=+1$ (for the dislocation close to the boundary) and $b_2$ randomly chosen between $+1$ and $-1$ at each run.
In the numerics, explicit formulae for the forces were implemented directly.
In this case, evaluating $c(\delta_0)$ in \eqref{eq:c} gives $c(\delta_0)\approx 116.7$, which makes the estimate \eqref{eq:ubht} invalid.
Nevertheless, at leading order, $T_{\mathrm{coll}}^{\partial\Omega}\leq 2\pi\delta_0^2\approx 0.2513$, which bounds all times computed, as it can be verified in the histogram plot in Figure \ref{yeah}; the peak is due to the fact that when $b_2=+1$, the dislocation $z_1$ is both attracted by the boundary and pushed towards it by $z_2$.

\begin{figure}[h]
\begin{center}
\includegraphics[scale=.6]{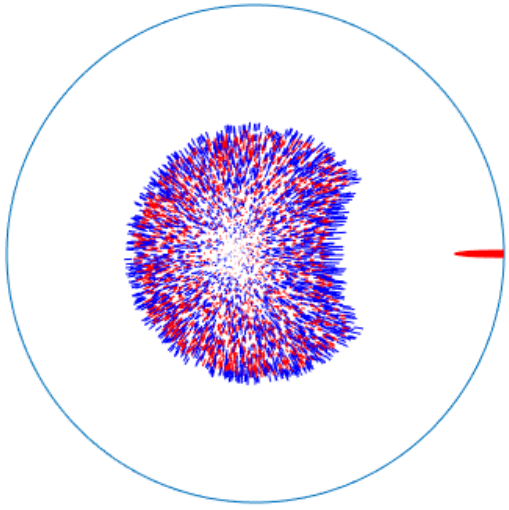}\qquad
\includegraphics[scale=.55]{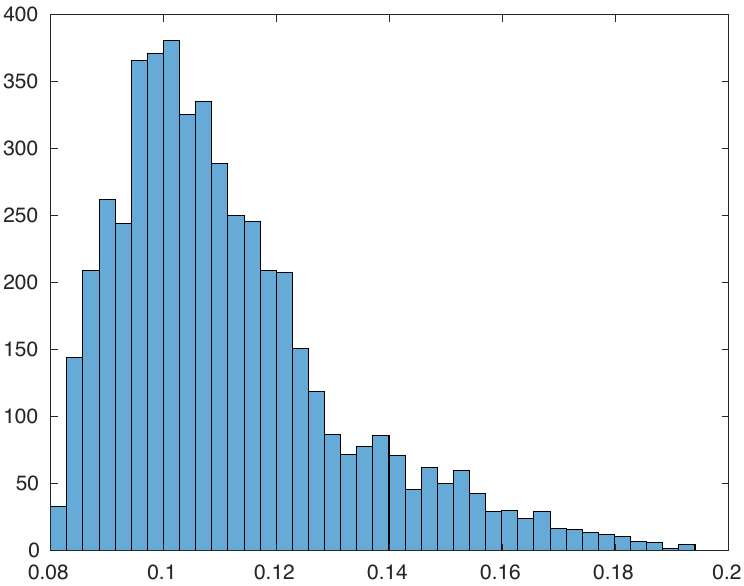}
\caption{5000 runs and histogram of hitting times. Red corresponds to $b_i=+1$, blue to $b_i=-1$.}
\label{yeah}
\end{center}
\end{figure}

Figure \ref{superyeah} shows plots of $80$ trajectories of one dislocation evolving in the square and in the cardioid. 
To numerically resolve $\nabla h_\Om$ in these cases, we used quadratic finite elements with a mesh generated by the package DistMesh, described in \cite{PS04}.
Both of these domains have an unstable equilibrium point at their centre, and initial conditions are chosen on a circle of radius $0.1$ centred at the equilibrium point. 
Due to the interaction with the boundary, the dislocation starts following a curved line and then hits the boundary perpendicularly (up to numerical artefacts), as indicated by \eqref{505} in Theorem \ref{thm:fatal} (see also estimate \eqref{estfatal}).
In the square, by symmetry, the dislocations starting on the diagonals move along them towards the corners.
We remark that the assumptions of Lemma \ref{th:h_derivative_bound}, which is crucial to prove Theorem \ref{thm:fatal}, explicitly exclude domains with corners, but to leading order, the conclusion appears to hold at smooth points of the boundary even in this case. 
We stress that the curved trajectories are a consequence of the interaction with the boundary and of its curvature.
\begin{figure}[h]
\begin{center}
\includegraphics[scale=.6]{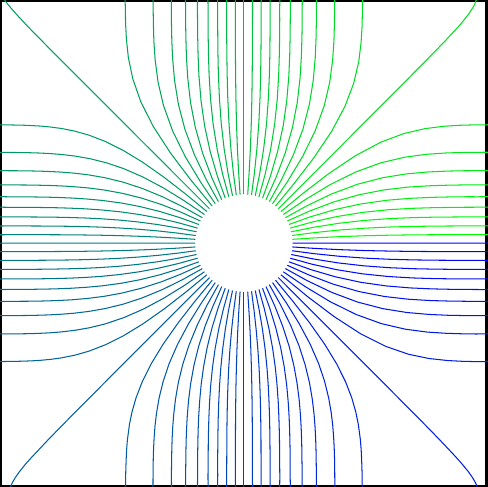}\qquad\qquad
\includegraphics[scale=.6]{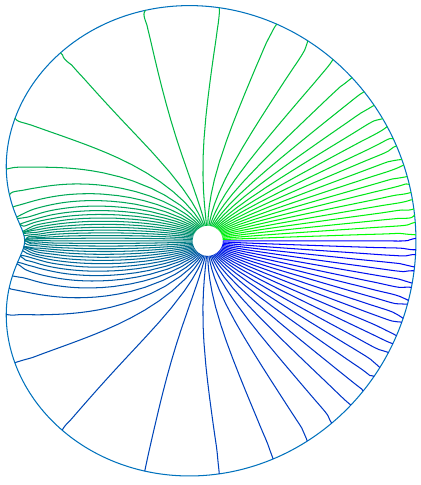}
\caption{Superposition of $80$ trajectories of one dislocation in the square and in the cardioid.}
\label{superyeah}
\end{center}
\end{figure}

It would be of interest to study the behaviour near non--smooth boundary points further, with a particular view to understanding the behaviour of dislocations near cracks.

\section{Concluding comments}\label{conclusion}
We have studied the qualitative behaviour of the dynamics of screw dislocations in two-dimensional domains, under the assumption of linear isotropic mobility.
We dealt with unconstrained dynamics which is the crucial first step towards considering more realistic choices of mobility, such as enforcing glide directions \cite{BFLM15,CG99} or other more general nonlinear mobilities \cite{H15}.
In these cases, the equations of motion \eqref{500} read
\begin{equation*}
\dot z_i(t)=\mathcal{M}[-\nabla_{z_i}\E_n(z_1(t),\ldots,z_n(t))],
\end{equation*}
where the mobility function $\mathcal{M}$ prescribes a law relating the Peach--Koehler force $f_i(z):=-\nabla_{z_i}\E_n(z)$ experienced by the dislocation sitting at $z_i$ to its velocity.

\subsection{On more general mobility functions}
In \cite{CG99}, a dissipative formulation is proposed to describe the motion of screw dislocations: dislocations are constrained to move along straight lines following the direction of maximal dissipation among finitely many \emph{glide directions}.
These are a finite set of lattice (unit) vectors $\mathcal{G}\subset\R^2$, such that $\mathrm{span}(\mathcal{G})=\R^2$ (which means that $\mathcal{G}$ contains at least two linearly independent vectors) and such that $-\mathcal{G}=\mathcal{G}$ (which means that $\mathcal{G}$ is symmetric under inversion).

In this case,  the velocity field is given by
\begin{equation*}
\mathcal{M}[f_i(z)]=(f_i(z)\cdot g_i)g_i,\quad\text{where } g_i\in\argmax_{g\in\mathcal{G}} \{f_i(z_1,\ldots,z_n){\cdot}g\},
\end{equation*}
from which it is clear that glide directions reduce the modulus of the velocity.
From elementary geometric considerations concerning scalar products in the plane, a dislocation moves fastest when there exists a glide direction $g$ aligned with the Peach--Koehler force, whereas it moves slowest when the Peach--Koehler force is aligned with the bisector of the largest angle among glide directions.
Using this facts, it can be checked that the qualitative behaviour remains as described in Theorem \ref{thm:boundary_collision} and Theorem \ref{thm:collision}.

More generally, we believe that the results contained in this paper are suitable to treat more general mobility functions $\mathcal{M}$ satisfying appropriate growth conditions at infinity; see \cite{H15} for an example of such a case.

\subsection{Review of achievements}\label{upbeat}
We focused on the interaction of one dislocation with the boundary and on the collision of two dislocations. 
In the former case, we have analytically shown that, if sufficiently close to the boundary, dislocations experience a force directed along the outward normal to the boundary at the nearest point, thus formalising the fact that free boundaries attract dislocations (see \cite{LMSZ} for the behaviour with different boundary conditions), and that a dislocation sufficiently close to the boundary collides with in finite time.
In the latter, we have proved that two dislocations of opposite Burgers moduli that are sufficiently close to each other collide.
In both cases, we have found an upper bound for the collision time in terms of the geometry of the initial configuration and we have given sufficient conditions under which no other collisions happen.
These sufficient conditions are encoded in inequalities \eqref{eq:time_comparison_bdry} and \eqref{928}, where the different scaling of the left hand sides in $\delta_0$ and $\zeta_0$, respectively, compared to the right hand sides may be heuristically viewed as a time--scale separation. 
Indeed, in the dynamics described by \eqref{500}, dislocations that are close to a blow--up event acquire infinite speed. 

Moreover, we validated our analytical results by devising numerical experiments to show their consistency. 
The output of the numerics is contained in \eqref{yeah}, where a plot of the dynamics for two dislocations in the disk is presented, together with a histogram of hitting times, which is consistent with the bound \eqref{eq:ubht} on $T_{\mathrm{coll}}^{\partial\Omega}$ provided in Theorem \ref{thm:boundary_collision}.
Finally, Figure \ref{superyeah} shows numerical experiments for different domains, namely a square and a cardioid, both having an unstable equilibrium at their centre.
We remark that the square domain does not satisfy the hypotheses of Theorem \ref{thm:boundary_collision}, because of the corners; nonetheless, the dynamics can be solved numerically.

\appendix
\section{Interaction functions for the interior and exterior of a disk}

\label{sec:disk_funcs}\label{appendix}
In this appendix, we recall the definitions of Green's functions on the
interior and exterior of the disk $B_\rhobar(0)$, and compute $k_{B_\rhobar(0)}$
and $h_{B_\rhobar(0)}$ in these cases. 

We start by recalling that if $x,y\in \overline{B_\rhobar(0)}$, then, defining $x^*:=\rhobar^2 x/|x|^{2}$,
\begin{equation}\label{eq:Green_func_ball_interior}
G_{B_\rhobar(0)}(x,y)=-\frac1{2\pi}\bg[\log|x-y|-\log\bg(\frac{|x||x^*-y|}{\rhobar}\bg)\bg].
\end{equation}
Applying \eqref{eq:Green} and \eqref{eq:h_Om}, we find that
\begin{equation*}
  k_{B_\rhobar(0)}(x,y) = \frac{1}{2\pi}\log\left(\frac{|x||x^*-y|}{\rhobar}\right)\quad
  \text{and}\quad h_{B_\rhobar(0)}(x) = \frac{1}{4\pi}\log\left(\frac{\rhobar^2-|x|^2}{\rhobar}\right).
\end{equation*}

The Green's function on the exterior of the ball may be computed by a
further circular reflection. Let $x,y\in 
B_\rhobar(0)^c$. By considering the
conformal change of coordinates $x^*:= \rhobar^2 x/|x|^2$ and $y^*:= \rhobar^2 y/|y|^2$, it
is straightforward to check that
\begin{equation*}
  G_{B_\rhobar(0)^c}(x,y) = G_{B_\rhobar(0)}(x^*,y^*).
\end{equation*}
Applying \eqref{eq:Green}, after some algebraic manipulation using the properties of the logarithm,
we find that
\begin{equation*}
\begin{split}
  k_{B_\rhobar(0)^c}(x,y)&=G_{B_\rhobar(0)}(x^*,y^*)+\frac{1}{2\pi}\log|x-y|\\
  &=\frac1{2\pi}\bg[\log|x-y|+\log\bg(\frac{|x^*||x-y^*|}{\rhobar}\bg)-\log|x^*-y^*|\bg]\\
  &=\frac{1}{2\pi}\log\left(\frac{|y||x-y^*|}{\rhobar}\right).
\end{split}
\end{equation*}
Consequently, \eqref{eq:h_Om} implies
$$h_{B_\rhobar(0)^c}(x) = \frac{1}{2\pi}\log\left(\frac{|x|^2-\rhobar^2}\rhobar\right).$$
By writing $|x| = \rhobar\pm d_1(x)$, 
we obtain 
\begin{align}
  h_{B_\rhobar(0)}(x) &= \frac1{2\pi}\log\bg|2d_1(x)-\frac{d_1^2(x)}\rhobar\bg|
  \quad\text{for any }x\in B_\rhobar(0),\text{ and}\label{eq:h_interior_ball}\\
  h_{B_\rhobar(0)^c}(x) &= \frac1{2\pi}\log\bg|2d_1(x)+\frac{d_1^2(x)}\rhobar\bg| \quad\text{for any }x\in B_\rhobar(0)^c.
  \nonumber
\end{align}
We remark that the former expression is a multiple of (2.3) in \cite{CF85}. 
In both cases, $h_\Om$ diverges logarithmically to $-\infty$ as $x$ approaches $\partial\Om$.

\bigskip

\section*{Acknowledgments}
T.H.\@ thanks Carnegie Mellon University, \'Ecole des Ponts, INRIA, and the University of Warwick, and M.M.\@ thanks SISSA and Technische Universit\"at M\"unchen, where this research was carried out.
Both authors are thankful to Irene Fonseca and Giovanni Leoni for suggesting the topic of research, and to Timothy Blass for helpful discussions.
M.M.\@ is a member of the Gruppo Nazionale per l'Analisi Matematica, la Probabilità e le loro Applicazioni (GNAMPA) of the Istituto Nazionale di Alta Matematica (INdAM).

The research of T.H.\@ was funded both by a public grant overseen by the French National Research Agency (ANR) as part of the \emph{Investissements d'Avenir} program (reference: ANR-10-LABX-0098), and by an Early Career Fellowship, awarded by the Leverhulme Trust.
The research of M.M.\@ was partially funded by the ERC Advanced grant \emph{Quasistatic and Dynamic Evolution Problems in Plasticity and Fracture} (Grant agreement no.: 290888) and by the ERC Starting grant \emph{High-Dimensional Sparse Optimal Control} (Grant agreement no.: 306274).

\bibliographystyle{alpha}

\bibliography{DislBDRY}

\begin{thebibliography}{CEHMR10}

\bibitem[ADLGP14]{ADLGP14}
R.~Alicandro, L.~De~Luca, A.~Garroni, and M.~Ponsiglione.
\newblock Metastability and dynamics of discrete topological singularities in
  two dimensions: a {$\Gamma$}-convergence approach.
\newblock {\em Arch. Ration. Mech. Anal.}, 214(1):269--330, 2014.

\bibitem[ADLGP16]{ADLGP16}
R.~Alicandro, L.~De~Luca, A.~Garroni, and M.~Ponsiglione.
\newblock Dynamics of discrete screw dislocations on glide directions.
\newblock {\em J. Mech. Phys. Solids}, 92:87--104, 2016.

\bibitem[BBH94]{BBH94}
F.~Bethuel, H.~Brezis, and F.~H\'elein.
\newblock {\em Ginzburg-{L}andau vortices}, volume~13 of {\em Progress in
  Nonlinear Differential Equations and their Applications}.
\newblock Birkh\"auser Boston, Inc., Boston, MA, 1994.

\bibitem[BFLM15]{BFLM15}
T.~Blass, I.~Fonseca, G.~Leoni, and M.~Morandotti.
\newblock Dynamics for systems of screw dislocations.
\newblock {\em SIAM J. Appl. Math.}, 75(2):393--419, 2015.

\bibitem[BM17]{BM17}
T.~Blass and M.~Morandotti.
\newblock Renormalized energy and peach-k\"ohler forces for screw dislocations
  with antiplane shear.
\newblock {\em J. Convex Anal.}, 24(2), 2017.
\newblock To appear.

\bibitem[BOS05]{BOS05}
F.~Bethuel, G.~Orlandi, and D.~Smets.
\newblock Collisions and phase-vortex interactions in dissipative
  {G}inzburg-{L}andau dynamics.
\newblock {\em Duke Math. J.}, 130(3):523--614, 2005.

\bibitem[BvMM16]{BvMM15}
G.~A. Bonaschi, P.~van Meurs, and M.~Morandotti.
\newblock Dynamics of screw dislocations: a generalised minimising-movements
  scheme approach.
\newblock {\em Eur. J. Appl. Math.}, pages 1--20, 2016.
\newblock To appear. DOI: 10.1017/S0956792516000462.

\bibitem[CEHMR10]{CEH10}
M.~Cannone, A.~El~Hajj, R.~Monneau, and F.~Ribaud.
\newblock Global existence for a system of non-linear and non-local transport
  equations describing the dynamics of dislocation densities.
\newblock {\em Arch. Ration. Mech. Anal.}, 196(1):71--96, 2010.

\bibitem[CF85]{CF85}
L.~A. Caffarelli and A.~Friedman.
\newblock Convexity of solutions of semilinear elliptic equations.
\newblock {\em Duke Math. J.}, 52(2):431--456, 1985.

\bibitem[CG99]{CG99}
P.~Cermelli and M.~E. Gurtin.
\newblock The motion of screw dislocations in crystalline materials undergoing
  antiplane shear: glide, cross-slip, fine cross-slip.
\newblock {\em Arch. Ration. Mech. Anal.}, 148(1):3--52, 1999.

\bibitem[CL05]{CL05}
P.~Cermelli and G.~Leoni.
\newblock Renormalized energy and forces on dislocations.
\newblock {\em SIAM J. Math. Anal.}, 37(4):1131--1160 (electronic), 2005.

\bibitem[FM09]{FM09}
N.~Forcadel and R.~Monneau.
\newblock Existence of solutions for a model describing the dynamics of
  junctions between dislocations.
\newblock {\em SIAM J. Math. Anal.}, 40(6):2517--2535, 2009.

\bibitem[Fri88]{Friedman}
A.~Friedman.
\newblock {\em Variational principles and free-boundary problems}.
\newblock Robert E. Krieger Publishing Co., Inc., Malabar, FL, second edition,
  1988.

\bibitem[GM10]{GM10}
A.~Ghorbel and R.~Monneau.
\newblock Well-posedness and numerical analysis of a one-dimensional non-local
  transport equation modelling dislocations dynamics.
\newblock {\em Math. Comp.}, 79(271):1535--1564, 2010.

\bibitem[GT01]{GT}
D.~Gilbarg and N.~S. Trudinger.
\newblock {\em Elliptic partial differential equations of second order}.
\newblock Classics in Mathematics. Springer-Verlag, Berlin, 2001.
\newblock Reprint of the 1998 edition.

\bibitem[Gus90]{G90}
B.~Gustafsson.
\newblock On the convexity of a solution of {L}iouville's equation.
\newblock {\em Duke Math. J.}, 60(2):303--311, 1990.

\bibitem[Hel14]{H14}
L.~L. Helms.
\newblock {\em Potential theory}.
\newblock Universitext. Springer, London, second edition, 2014.

\bibitem[HL82]{HL1982}
J.P. Hirth and J.~Lothe.
\newblock {\em Theory of Dislocations}.
\newblock Krieger Publishing Company, 1982.

\bibitem[Hud17]{H15}
T.~Hudson.
\newblock Upscaling a model for the thermally-driven motion of screw
  dislocations.
\newblock {\em Arch. Ration. Mech. Anal.}, 224(1):291--352, 2017.

\bibitem[JS98]{JS98}
R.~L. Jerrard and H.~M. Soner.
\newblock Dynamics of {G}inzburg-{L}andau vortices.
\newblock {\em Arch. Rational Mech. Anal.}, 142(2):99--125, 1998.

\bibitem[KP81]{KP81}
S.~G. Krantz and H.~R. Parks.
\newblock Distance to {$C^k$} hypersurfaces.
\newblock {\em Journal of Differential Equations}, 40(1):116 -- 120, 1981.

\bibitem[Lin96]{L96}
F.-H. Lin.
\newblock Some dynamical properties of {G}inzburg-{L}andau vortices.
\newblock {\em Comm. Pure Appl. Math.}, 49(4):323--359, 1996.

\bibitem[LMSZ16]{LMSZ}
I.~Lucardesi, M.~Morandotti, R.~Scala, and D.~Zucco.
\newblock Confinement of dislocations inside a crystal with a prescribed
  external strain.
\newblock {\em arXiv:1610.0685}, 2016.
\newblock submitted.

\bibitem[MP12]{MP12}
R.~Monneau and S.~Patrizi.
\newblock Homogenization of the {P}eierls-{N}abarro model for dislocation
  dynamics.
\newblock {\em J. Differential Equations}, 253(7):2064--2105, 2012.

\bibitem[Oro34]{Orowan34}
E.~Orowan.
\newblock Zur {K}ristallplastizit{\"a}t. {III}.
\newblock {\em Zeitschrift f{\"u}r Physik}, 89:634--659, 1934.

\bibitem[Pol34]{Polanyi34}
M.~Polanyi.
\newblock {\"U}ber eine {A}rt {G}itterst{\"o}rung, die einen {K}ristall
  plastisch machen k{\"o}nnte.
\newblock {\em Zeitschrift f{\"u}r Physik}, 89:660--664, 1934.

\bibitem[PS04]{PS04}
P.-O. Persson and G.~Strang.
\newblock A simple mesh generator in {M}atlab.
\newblock {\em SIAM Rev.}, 46(2):329--345, 2004.

\bibitem[Ser07]{S07}
S.~Serfaty.
\newblock Vortex collisions and energy-dissipation rates in the
  {G}inzburg-{L}andau heat flow. {II}. {T}he dynamics.
\newblock {\em J. Eur. Math. Soc. (JEMS)}, 9(3):383--426, 2007.

\bibitem[SS03]{SS03}
E.~Sandier and S.~Serfaty.
\newblock Ginzburg-{L}andau minimizers near the first critical field have
  bounded vorticity.
\newblock {\em Calc. Var. Partial Differential Equations}, 17(1):17--28, 2003.

\bibitem[SS07]{SS07}
E.~Sandier and S.~Serfaty.
\newblock {\em Vortices in the magnetic {G}inzburg-{L}andau model}, volume~70
  of {\em Progress in Nonlinear Differential Equations and their Applications}.
\newblock Birkh\"auser Boston, Inc., Boston, MA, 2007.

\bibitem[Tay34]{Taylor34}
G.~I. Taylor.
\newblock The mechanism of plastic deformation of crystals. {P}art {I}.
  {T}heoretical.
\newblock {\em Proceedings of the Royal Society of London. Series A, Containing
  Papers of a Mathematical and Physical Character}, 145(855), 1934.

\bibitem[vMM14]{vMM14}
P.~van Meurs and A.~Muntean.
\newblock Upscaling of the dynamics of dislocation walls.
\newblock {\em Adv. Math. Sci. Appl.}, 24(2):401--414, 2014.

\bibitem[Vol07]{V1907}
V.~Volterra.
\newblock Sur l'\'equilibre des corps \'elastiques multiplement connexes.
\newblock {\em Annales scientifiques de l'\'Ecole Normale Sup\'erieure},
  24:401--517, 1907.

\end{thebibliography}

\end{document}